\documentclass[reqno]{amsart}
\usepackage{fullpage,amssymb,amsmath,verbatim,amstext,eucal}

\newtheorem{lemma}{Lemma}[subsection] 

\newtheorem{proposition}[lemma]{Proposition}
\newtheorem{theorem}[lemma]{Theorem}

\newtheorem{corollary}[lemma]{Corollary}

\newcommand{\FLEX}{\relax}
\newcommand{\flex}[1]{\renewcommand{\FLEX}{#1}}
\newtheorem{flexthm}[lemma]{\FLEX}
\newenvironment{flexstate}[2]{\flex{#1}\begin{flexthm}[#2]}{\end{flexthm}}

\theoremstyle{definition}
\newtheorem{definition}[lemma]{Definition}
\newtheorem{example}[lemma]{Example}

\newenvironment{remark}[1]{\refstepcounter{lemma}%
\vskip 5pt \par\noindent {\bf #1\ \thelemma .}}{\vskip 5pt \par}

\newenvironment{remark*}[1]{\par \vskip 5pt \noindent
{\bf #1.}}{\vskip 5pt \par}

\newlength{\dqlength}

\newcommand{\alg}{{\rm Alg}\,}

\newcommand{\bh}{\ensuremath{{\mathcal B}({\mathcal H})}}

\newcommand{\cstar}{\hbox{$C^*$}}
\newcommand{\cstaralg}{$C^*$-algebra}

\providecommand{\dual}[1]{\ensuremath{#1^{\#}}}

\newcommand{\dom}{\operatorname{dom}}

\newcommand{\dstext}[1]{\quad\text{#1}\quad}
\newcommand{\eps}{\ensuremath{\varepsilon}}

\newcommand{\innerprod}[1]{\left\langle #1\right\rangle}

\newcommand{\join}{\bigvee}

\newcommand{\lat}{{\rm Lat}\,}

\newcommand{\meet}{\bigwedge}

\newcommand{\norm}[1]{\left\|{#1}\right\|}

\providecommand{\qed}%
{\hfill \vrule height5pt width4pt depth1pt \vspace{+2.00ex}}

\newcommand{\ran}{\operatorname{range}}

\newcommand{\spn}{\operatorname{span}}
\newcommand{\supp}{\operatorname{supp}}

\newcommand{\bbC}{{\mathbb{C}}}

\newcommand{\bbN}{{\mathbb{N}}}

\newcommand{\bbR}{{\mathbb{R}}}

\newcommand{\bbT}{{\mathbb{T}}}

\renewcommand{\AA}{{\mathbf{A}}}

\newcommand{\MM}{{\mathbf{M}}}

\newcommand{\A}{{\mathcal{A}}}

\newcommand{\B}{{\mathcal{B}}}
\newcommand{\C}{{\mathcal{C}}}

\newcommand{\D}{{\mathcal{D}}}

\newcommand{\E}{{\mathcal{E}}}

\newcommand{\G}{{\mathcal{G}}}
\renewcommand{\H}{{\mathcal{H}}}

\renewcommand{\L}{{\mathcal{L}}}
\newcommand{\M}{{\mathcal{M}}}
\newcommand{\N}{{\mathcal{N}}}

\newcommand{\R}{{\mathcal{R}}}
\renewcommand{\S}{{\mathcal{S}}}

\newcommand{\U}{{\mathcal{U}}}

\newcommand{\Z}{{\mathcal{Z}}}

\newcommand{\fA}{{\mathfrak{A}}}
\newcommand{\fB}{{\mathfrak{B}}}

\newcommand{\fD}{{\mathfrak{D}}}

\newcommand{\fH}{{\mathfrak{H}}}

\newcommand{\fL}{{\mathfrak{L}}}

\newcommand{\fS}{{\mathfrak{S}}}

\newcommand{\fT}{{\mathfrak{T}}}

\newcommand{\cstardiag}{\cstar-diagonal}

\newcommand{\bimod}{\operatorname{bimod}}

\newcommand{\mean}{\mathop{
     \mathchoice{\vcenter{\hbox{\huge
           $\Lambda$}}}{\Lambda}{\Lambda}{\Lambda}}
     \displaylimits}

\newcommand{\fn}{\mathfrak{n}}

\newcommand{\pil}{\pi_\ell}
\newcommand{\pir}{\pi_r}
\newcommand{\Nsf}{\N_{\text{sf}}}

\begin{document}

\title{Bimodules over Cartan MASAs in von Neumann Algebras, Norming
  Algebras, and Mercer's Theorem}
\author{Jan Cameron}
\address{Dept. of Mathematics\\ Vassar College\\ Poughkeepsie, NY\\ 12604}
\email{jacameron@vassar.edu}
\author[D.R. Pitts]{David R. Pitts}

\address{Dept. of Mathematics\\ 
University of Nebraska-Lincoln\\ Lincoln, NE\\ 68588-0130}
\email{dpitts2@math.unl.edu}
\author{Vrej Zarikian}
\thanks{Zarikian was partially supported by Nebraska IMMERSE}
\address{Dept. of Mathematics\\ U. S. Naval Academy\\ Annapolis, MD\\ 21402}
\email{zarikian@usna.edu}


\begin{abstract}
  In a 1991 paper, R. Mercer asserted that a Cartan bimodule
  isomorphism between Cartan bimodule algebras $\A_1$ and $\A_2$
  extends uniquely to a normal $*$-isomorphism of the von Neumann
  algebras generated by $\A_1$ and
  $\A_2$~\cite[Corollary~4.3]{MercerIsIsCaBiAl}.  Mercer's argument
  relied upon the Spectral Theorem for Bimodules of Muhly, Saito and
  Solel~\cite[Theorem~2.5]{MuhlySaitoSolelCoTrOpAl}.  Unfortunately,
  the arguments in the literature supporting
  \cite[Theorem~2.5]{MuhlySaitoSolelCoTrOpAl} contain gaps, and hence
  Mercer's proof is incomplete.

  In this paper, we use the outline
  in~\cite[Remark~2.17]{PittsNoAlAuCoBoIsOpAl} to give a proof of
  Mercer's Theorem under the additional hypothesis that the given
  Cartan bimodule isomorphism is $\sigma$-weakly continuous.  Unlike
  the arguments contained in
  \cite{MercerIsIsCaBiAl,MuhlySaitoSolelCoTrOpAl}, we avoid the use of
  the Feldman-Moore machinery from \cite{FeldmanMooreErEqReII}; as a
  consequence, our proof does not require the von Neumann algebras
  generated by the algebras $\A_i$ to have separable preduals.  This
  point of view also yields some insights on the von Neumann
  subalgebras of a Cartan pair $(\M,\D),$ for instance, a
  strengthening of a result of Aoi \cite{AoiCoEqSuInSu}.
  
  We also examine the relationship between various topologies on a von
  Neumann algebra $\M$ with a Cartan MASA $\D$.  This provides the
  necessary tools to parametrize the family of Bures-closed bimodules
  over a Cartan MASA in terms of projections in a certain abelian von
  Neumann algebra; this result may be viewed as a weaker form of the
  Spectral Theorem for Bimodules, and is a key ingredient in the proof
  of our version of Mercer's theorem.  Our results lead to a notion of
  spectral synthesis for $\sigma$-weakly closed bimodules appropriate
  to our context, and we show that any von Neumann subalgebra of $\M$
  which contains $\D$ is synthetic.

  We observe that a result of Sinclair and Smith shows that any Cartan
  MASA in a von Neumann algebra is norming in
  the sense of Pop, Sinclair and Smith.

\end{abstract}

\maketitle

\vskip 12pt

\section{Background and Preliminaries}

\subsection{Background}
The following appears in a 1991 paper of R. Mercer:

\begin{flexstate}{Assertion}{{\cite[Corollary~4.3]{MercerIsIsCaBiAl}}} 
\label{Mercer}
  For $i= 1, 2$, let $\M_i$ be a von Neumann algebra with separable
  predual and let $\D_i\subseteq \M_i$ be a Cartan MASA.  Suppose
  $\A_i$ is a $\sigma$-weakly closed subalgebra of $\M_i$ which
  contains $\D_i$ and which generates $\M_i$ as a von Neumann algebra.

  If $\theta:\A_1\to \A_2$ is an isometric algebra isomorphism such
  that $\theta(\D_1) = \D_2$, then $\theta$ extends to a von Neumann
  algebra isomorphism $\overline\theta:\M_1 \to \M_2$. Furthermore, if
  $\M_i$ is identified with its Feldman-Moore representation, so $\M_i
  \subseteq \B(L^2(R_i))$, then $\overline\theta$ may be taken to be a
  spatial isomorphism.
\end{flexstate}

Mercer's argument supporting this assertion relies upon the Spectral
Theorem for Bimodules of Muhly, Saito and
Solel~\cite[Theorem~2.5]{MuhlySaitoSolelCoTrOpAl}.  The purpose
of~\cite[Theorem~2.5]{MuhlySaitoSolelCoTrOpAl} is to characterize
$\sigma$-weakly closed bimodules over a Cartan MASA in terms of
measure-theoretic data.  We know of
two articles claiming to prove this characterization: the original
paper~\cite{MuhlySaitoSolelCoTrOpAl} and another paper of Mercer,
see~\cite[Theorem~5.1]{MercerBiOvCaSu}.  Unfortunately, the proofs in
both articles contain gaps, so the validity
of~\cite[Theorem~2.5]{MuhlySaitoSolelCoTrOpAl} in general is
uncertain.  However, for $\sigma$-weakly closed bimodules over a
Cartan MASA in a hyperfinite von Neumann algebra, the Spectral Theorem
for Bimodules follows from a more general result of Fulman,
see~\cite[Theorem~15.18]{FulmanCrPrvNAlEqReThSu}.

The paper of Aoi~\cite[pages 724--725]{AoiCoEqSuInSu} gives a
discussion of the gap in the proof presented
in~\cite[Theorem~2.5]{MuhlySaitoSolelCoTrOpAl}.  On the other hand, Mercer's
argument (see the proof of \cite[Theorem~5.1]{MercerBiOvCaSu}) claims
that if $(\M,\D)$ is a pair consisting of a separably-acting von
Neumann algebra $\M$ and a Cartan MASA $\D$, and $\S\subseteq \M$ is a
$\sigma$-weakly closed subspace, then $\S$ is closed in the relative
$L^2$ topology.  (This is the topology arising from the norm, $\M \ni
T \mapsto \sqrt{\omega(E(T^*T))}$, where $\omega$ is a fixed faithful
normal state on $\D$ and $E:\M\rightarrow \D$ is the faithful normal
conditional expectation.)  The following example, from Roger Smith,
shows this statement is false.

\begin{example}
  Let $\M = \D =L^\infty[0,1]$ where the measure is Lebesgue measure.
  In this case, the $L^2$ topology on $\M$ is the relative topology on
  $\M$ arising from viewing $\M$ as a subspace of $L^2[0,1]$.  Since
  $\M_*$ may be identified with $L^1[0,1]$, the linear functional
  $\phi$ on $\M$ given by
\[
    \phi(f) := \int_0^1 f(x) x^{-3/4}\, dx
\]
is $\sigma$-weakly continuous.  Let $\S := \ker\phi$.  Then $\S$ is
$\sigma$-weakly closed.  But $\phi$ is not continuous with respect to
the $L^2$-norm, so $\S$ is not $L^2$-closed \cite[Theorem
3.1]{ConwayCoFuAn}.
\end{example}

Because of these issues, the question of whether
Assertion~\ref{Mercer} is correct in general arises.  It is
interesting that when Assertion~\ref{Mercer} is valid, $\theta$ is
necessarily $\sigma$-weakly continuous.  While Mercer did not
explicitly assume $\theta$ is $\sigma$-weakly continuous (or
continuous with respect to another appropriate topology) in his
assertion, he does implicitly make one.  Indeed, Mercer's argument for
Assertion~\ref{Mercer} relies
upon~\cite[Proposition~2.2]{MercerIsIsCaBiAl}, and the first paragraph
of the proof of that proposition implicitly assumes a continuity
hypothesis.  Thus, the statement of Assertion~\ref{Mercer} appearing
in~\cite{MercerIsIsCaBiAl} should also include an appropriate
continuity assumption.

A principal goal of this paper is to provide a proof of
Assertion~\ref{Mercer}, under the additional hypothesis that $\theta$
is $\sigma$-weakly continuous, which does not use the Spectral Theorem
for Bimodules.  Our argument uses the notion of norming algebras and
follows the outline given in \cite[Remark
2.17]{PittsNoAlAuCoBoIsOpAl}.  Unlike Mercer's original statement, we
do not require that $\M$ have separable predual.  We shall require an
understanding of two topologies on $\M$, the Bures and $L^2$
topologies.  As a consequence of this analysis, we obtain
Theorem~\ref{newSTB}, the Spectral Theorem for Bures Closed Bimodules,
where the bimodules characterized are those which are closed in the
Bures (or, equivalently, the $L^2$) topology rather than the
$\sigma$-weak topology.  Instead of using measure theoretic data to
characterize Bures closed bimodules, our characterization uses
projections in a certain abelian von Neumann algebra constructed from
the Cartan MASA $\D$ and $\M$.  This leads to a notion of synthesis
similar to that found in Arveson's seminal
paper~\cite{ArvesonOpAlInSu}, but appropriate to our context.  When
$\A\subseteq \M$ is a von Neumann algebra containing $\D$, we give a
new proof, and a strengthening, of a result of Aoi
\cite{AoiCoEqSuInSu}, which shows that $\D$ is a Cartan MASA in $\A$
and establishes the existence of a conditional expectation from $\M$
onto $\A.$ Our methods also show that any von Neumann subalgebra of
$\M$ containing $\D$ is Bures closed, from which it follows that the
class of von Neumann subalgebras of $\M$ which contain $\D$ is a class
of $\D$-bimodules which satisfy synthesis and for which the conclusion of
\cite[Theorem~2.5]{MuhlySaitoSolelCoTrOpAl} is valid.

We are grateful to the referee of a previous version of this paper for
alerting us to the issues involving the Spectral Theorem for Bimodules
and to Paul Muhly for the references to the papers of Aoi and Fulman.

We also wish to acknowledge our indebtedness to the very interesting
papers of Muhly-Saito-Solel~\cite{MuhlySaitoSolelCoTrOpAl} and
Mercer~\cite{MercerBiOvCaSu} discussed above.  Many of the ideas found in
those papers provide techniques for the analysis of
bimodules in our context.  We utilized several of the tools in those
papers and the present paper would not have been written without them.

\subsection{Some General Notation}

Because we shall be dealing with certain nonselfadjoint algebras, we
use $\dual{X}$ for the dual of the Banach space $X$; likewise, when
$X$ is a complex vector space and $\tau$ is a locally convex topology
on $X$, $\dual{(X,\tau)}$ will denote its dual space.

For any unital \cstaralg\ $\C$ containing a unital abelian $C^*$-subalgebra
$\D$, let
\[ \N(\C,\D) := \{v \in \C: v^*\D v \cup v\D v^* \subseteq \D\}.
\] An element $v \in \N(\C,\D)$ is a \textit{normalizer} of $\D$.
Finally, if $v \in \N(\C,\D)$ is a partial isometry, then we say that
$v$ is a \emph{groupoid normalizer} of $\D$, and write $v \in
\G\N(\C,\D)$.

\begin{lemma}\label{polarnormal} Let $\M$ be a von Neumann algebra, 
let $\D \subseteq \M$ be an abelian von Neumann subalgebra (with the
same unit) and let $\S \subseteq \M$ be a
$\sigma$-weakly closed $\D$-bimodule.  Given $v \in \S\cap\N(\M,\D)$, 
let $v = u|v|$ be the polar
decomposition of $v$.  Then $u \in \S\cap\G\N(\M,\D)$.
\end{lemma}

\begin{proof} The statement is trivial if $v = 0$, so assume $v \neq
0$.  Since $v \in \N(\M,\D)$, $v^* I v \in \D$, so $|v| \in \D$.  Let
$S$ be the spectral measure for $|v|$.  For $0 < \eps < \norm{v}$, let
$f_\eps(t) = t^{-1}\chi_{[\eps,\infty)}(t)$ and $P_\eps =
S([\eps,\norm{v}])$.  Then $|v|f_\eps(|v|) = P_\eps$, so $vf_\eps(|v|)
= uP_\eps$ converges $\sigma$-strong-$*$ to $u$ as $\eps \rightarrow
0$.  Also, $vf_\eps(|v|) \in \N(\M,\D)$ with $\norm{vf_\eps(|v|)} \leq
1$. Since multiplication on bounded sets is jointly continuous in the
$\sigma $-strong topology, we conclude that $u \in \G\N(\M,\D)$.

Since
$v \in \S$, $u|v|^n = v|v|^{n-1} \in \S$ for all $n \in \bbN$, which
implies $u|v|^{1/n} \in \S$ for all $n \in \bbN$.  But $u|v|^{1/n}
\stackrel{\sigma\text{-weak}}{\longrightarrow} uu^*u = u$, so $u \in \S$.
\end{proof}

\begin{definition}
A MASA $\D$ in a von Neumann algebra $\M$ is called a \textit{Cartan MASA} if
there is a faithful, normal conditional expectation $E:\M \rightarrow
\D$ and $\spn\{U \in \M: U \text{ is unitary and } U\D U^* = \D\}$ is
$\sigma$-weakly dense in $\M$.  We will call the pair $(\M,\D)$ a
\textit{Cartan pair}.  
\end{definition}

\begin{remark}{Standing Assumption} Unless explicitly stated to the
  contrary, throughout this paper, $\M$ will denote a von Neumann
  algebra with a Cartan MASA $\D$.  \end{remark}

\subsection{Bimodules and Normalizers}

We now give some properties of the expectation $E$, and use them to
show that bimodules often contain a rich supply of normalizers.  We
require some notation.  Recall that any discrete abelian group $G$ has
an invariant mean $\mean$.  This means that $\mean$ is a state on
$\ell^\infty(G)$ such that for any $h \in G$ and $F \in
\ell^{\infty}(G)$, $\mean(F) = \mean(F_h)$, where $F_h(g) =
F(gh^{-1})$.  We will usually write, $\mean_{g \in G} F(g)$ instead of
$\mean(F)$.  We will always assume that $\mean$ has the additional
property that it is invariant under inversion, that is,
\[ \mean_{g \in G} F(g) = \mean_{g \in G} F(g^{-1});
\] this can be achieved by replacing $\mean$ if necessary with
$\tilde{\mean}$, where $\tilde{\mean}_{g \in G} F(g) = \mean_{g \in G}
\frac{F(g) + F(g^{-1})}{2}$.

We now require two lemmas, the first of which is standard.
Throughout, when $\C$ is a unital \cstaralg,  $\U(\C)$ denotes the
unitary group of $\C$.

\begin{lemma}\label{invmean} Let $X$ be a Banach space and let $\mean$
be an invariant mean on the (discrete) group $\U(\D)$.  Suppose that
$f:\U(\D) \rightarrow \dual{X}$ is a bounded function.  Then there
exists $T \in \overline{\text{co}}^{\text{weak-}*}\{f(U): U \in
\U(\D)\}$ such that for every $x \in X$,
\[ \innerprod{x,T} =\mean_U \innerprod{x,f(U)}.
\]
\end{lemma}

\begin{proof} The existence of $T$ follows from the fact that the map
$X \ni x \mapsto \mean_U \innerprod{x,f(U)}$ is a bounded linear
functional on $X$.  For every $x \in X$, $\innerprod{x,T}$
belongs to the closed convex hull of $\{\innerprod{x,f(U)}: U \in
\U(\D)\}$.  So a separation theorem shows that $T \in
\overline{\text{co}}^{\text{weak-}*}\{f(U): U \in \U(\D)\}$.
\end{proof}

\begin{remark}{Notation} In the setting of Lemma~\ref{invmean}, we
write $T := \mean_U f(U)$.
\end{remark}

The following well-known fact appears
as~\cite[Theorem~6.2.1]{ArvesonAnOpAl}.  Since it will be useful in
the sequel, we include a proof for the convenience of the reader.

\begin{lemma}\label{uniquece} For $T \in \M$,
\[ E(T) = \mean_{U\in\U(\D)} UTU^*
\] and
\[ \{E(T)\} = \D \cap
\overline{\text{co}}^{\sigma\text{-weak}}\{UTU^*: U \in \U(\D)\}.
\]
\end{lemma}

\begin{proof} For $T \in \M$, set $E_1(T) = \mean_{U \in \U(\D)}
UTU^*$.  Given $\rho \in \M_*$, and $W\in\U(\D)$, we have
\begin{align*} \rho(WE_1(T)) &= \mean_{U \in \U(\D)} \rho(WUTU^*)\\ &=
\mean_{U \in \U(\D)} \rho((WU)T(WU)^*W)\\ &= \mean_{U \in \U(\D)}
\rho(UTU^*W)\\ &= \rho(E_1(T)W).
\end{align*} Therefore $E_1(T)$ commutes with $\U(\D)$.  But $\D$ is
the linear span of $\U(\D)$, so $E_1(T) \in \D' \cap \M$.  Since $\D$
is a MASA in $\M$, $E_1(T) \in \D$.  The normality of $E$ and the fact
that $E(UTU^*) = E(T)$ for each $U \in \U(\D)$ yield
\begin{align*} \{E_1(T)\} &\subseteq \D \cap
\overline{\text{co}}^{\sigma\text{-weak}}\{UTU^*: U \in \U(\D)\}\\ &=
E(\D \cap \overline{\text{co}}^{\sigma\text{-weak}}\{UTU^*: U \in
\U(\D)\})\\ &\subseteq
E(\overline{\text{co}}^{\sigma\text{-weak}}\{UTU^*: U \in \U(\D)\})\\
&= \{E(T)\}.
\end{align*}
\end{proof}

The following result, together with Lemma \ref{polarnormal}, shows that any $\D$-bimodule in $\M$ which is closed in an
appropriate topology contains an abundance of groupoid normalizers.
The technique used here has been employed previously in several
articles, for example,
see~\cite[Proposition~4.4]{MuhlyQiuSolelCoNuSpSuOpAl}
or~\cite[Proposition~3.10]{DonsigPittsCoSyBoIs}.

\begin{proposition}\label{weaknormalizerexists} Let $\S \subseteq \M$
be a $\sigma$-weakly closed $\D$-bimodule.  If $v \in \N(\M,\D)$ and
$T \in \S$, then $vE(v^*T) \in \S$, and when $T \neq 0$, $v \in
\N(\M,\D)$ may be chosen so that $vE(v^*T) \neq 0$. In particular, if
$\S$ is non-zero, then $(\S \backslash \{0\}) \cap \N(\M,\D) \neq
\emptyset$.
\end{proposition}

\begin{proof} If $v \in N(\M,\D)$ and $T \in \S$, then
Lemma~\ref{uniquece} shows that
\[ \{vE(v^*T)\} \subseteq
v\,\overline{co}^{\sigma\text{-weak}}\{Uv^*TU^*: U \in \U(\D)\} =
\overline{co}^{\sigma\text{-weak}}\{(vUv^*)TU^*: U \in \U(\D)\}
\subseteq \S
\] (because $vUv^*\in \D$).

If $T \in \M$ satisfies $E(v^*T) = 0$ for every $v \in \N(\M,\D)$,
then $T = 0$.  Indeed, for every $x \in \spn\N(\M,\D)$, $E(x^*T) = 0$.
By normality of $E$, we conclude that $E(T^*T) = 0$.  As $E$ is
faithful, $T = 0$.

If $0 \neq T \in \M$ and $v \in \N(\M,\D)$ satisfies $E(v^*T)
\neq 0$, then $vE(v^*T) \neq 0$.  To see this, argue by contradiction.
If $vE(v^*T) = 0$, then $(v^*v)^nE(v^*T) = 0$ for every $n \in \bbN$.
Therefore, $(v^*v)^{1/n}E(v^*T) = 0$ for every $n \in \bbN$.  But
\[ 0 \neq E(v^*T) = \lim_{n \to \infty} E((v^*v)^{1/n}v^*T) = \lim_{n
\to \infty} (v^*v)^{1/n}E(v^*T) = 0,
\] which is absurd.  Thus $vE(v^*T)\neq 0$, and the proof is complete.
\end{proof}

We now give a slight generalization of a result appearing
in~\cite{MercerBiOvCaSu}.  We use it throughout the paper, often
without explicit mention.  We include the proof because it seems
novel.

\begin{lemma}[{\cite[Lemma~2.1]{MercerBiOvCaSu}}]\label{conjbyv} Let
$v\in\N(\M,\D)$.  Then for every $x \in \M$,
\[ E(v^*xv) = v^*E(x)v.
\]
\end{lemma}

\begin{proof} We prove this in several steps.

\textit{Step 1:} First, assume that $v$ is a unitary normalizer.
Since $v^*\U(\D)v = \U(\D)$, Lemma~\ref{uniquece} gives
\begin{align*} \{E(v^*xv)\} &=
\overline{co}^{\sigma\text{-weak}}\{U^*v^*xvU: U \in \U(\D)\} \cap
\D\\ &= \overline{co}^{\sigma\text{-weak}}\{v^*(vU^*v^*)x(vUv^*)v: U
\in \U(\D)\} \cap \D\\ &=
[v^*\left(\overline{co}^{\sigma\text{-weak}}\{(vU^*v^*)x(vUv^*): U \in
\U(\D)\}\right)v] \cap v^*\D v\\ &=
v^*[\left(\overline{co}^{\sigma\text{-weak}}\{(vU^*v^*)x(vUv^*): U \in
\U(\D)\}\right) \cap \D]v\\ &= \{v^*E(x)v\}.
\end{align*} Thus the lemma holds in this case.

\textit{Step 2:} Next, assume $v$ is a partial isometry.  Then
\[ V := \begin{pmatrix} v & (I - vv^*)\\ (I - v^*v) &
v^* \end{pmatrix}
\] is a unitary element of $M_2(\M) = \M \otimes M_2(\bbC)$.  Let
\[ D_2(\D) := \left\{\begin{pmatrix} d_1 & 0\\ 0 & d_2 \end{pmatrix} :
d_i \in \D\right\}.
\] Then $(M_2(\M),D_2(\D))$ is a Cartan pair, and the conditional
expectation is the map $E_2$ given by
\[ M_2(\M) \ni \begin{pmatrix} y_{11} & y_{12}\\ y_{21} &
y_{22} \end{pmatrix} \mapsto \begin{pmatrix} E(y_{11}) & 0\\ 0 &
E(y_{22}) \end{pmatrix}.
\] A simple calculation using the fact that $vv^*, v^*v\in \D$ shows
that $V \in \N(M_2(\M),D_2(\D))$.  By Step 1, we have, for $X
= \begin{pmatrix} x & 0\\ 0 & 0 \end{pmatrix}$, $E_2(V^*XV) =
V^*E_2(X)V$.  The equality of the upper-left corner entries of these
matrices yields $E(v^*xv) = v^*E(x)v$.

\textit{Step 3:} Finally, assume that $v$ is a general normalizer.
Let $v = u|v|$ be the polar decomposition of $v$.  Then $u$ is a
partial isometry normalizer, by Lemma~\ref{polarnormal}.  Since $|v|
\in \D$, we have
\[ E(v^*xv) = |v|E(u^*xu)|v| =|v|u^*E(x)u|v| = v^*E(x)v.
\]
\end{proof}

\subsection{A MASA}

Here we show that when $(\M,\D)$ is in the standard form arising from
a suitable weight, then the von Neumann algebra generated by $\D$ and
$\D'$ is a MASA.  As a corollary, we show that $\D$ norms $\M$, in the
sense of Pop-Sinclair-Smith~\cite{PopSinclairSmithNoC*Al}.  Note that
these observations are implicit in ~\cite{SinclairSmithHoCoVNAlCaSu}
when the von Neumann algebra $\M$ is assumed to be finite and have
separable predual.

Fix a faithful, normal, semifinite weight $\phi$ on $\M$ such that
$\phi \circ E =\phi$.  (If $\omega$ is a faithful, normal, semifinite
weight on $\D$, then $\phi=\omega \circ E$ is such a weight on $\M$,
see~\cite[Proposition~IX.4.3]{TakesakiThOpAlII}.)  We freely use
notation from \cite{TakesakiThOpAlII}: in particular,
\[ \fn_\phi := \{T \in \M: \phi(T^*T) <\infty\},
\] and $(\pi_\phi,\fH_\phi,\eta_\phi)$ is the semi-cyclic
representation associated to $\phi$.
(See~\cite[VII.1]{TakesakiThOpAlII} for more details.)  Since
$E(T)^*E(T)\leq E(T^*T)$ for every $T \in \M$, we have $E(\fn_\phi) =
\fn_\phi \cap \D$.

\begin{lemma}\label{sotdense} Let $\Gamma = \{d \in \fn_\phi \cap \D:
0 \leq d \leq I\}$ and view $\Gamma$ as a net indexed by itself. Then
for $x \in \fn_\phi$, $\lim_{d \in \Gamma} \eta_\phi(xd) =
\eta_\phi(x)$.
\end{lemma}

\begin{proof} Let $S$ be the spectral measure for $E(x^*x)$, and let
$\mu$ be the (finite) Borel measure on $[0,\infty)$ defined by $\mu(A)
= \phi(E(x^*x)S(A))$.  Then $\lim_{t \to 0} \mu([0,t)) = \mu(\{0\}) =
0$, so given $\eps > 0$ we may find $t > 0$ so that $\mu([0,t)) <
\eps^2$.  Since $tS([t,\infty)) \leq E(x^*x)$, we obtain $p :=
S([t,\infty)) \in \Gamma$.  For $d \in \Gamma$ with $d \geq p$, we
have,
\[ \norm{\eta_\phi(x) - \eta_\phi(xd)}^2 = \phi(E(x^*x)(I-d)^2) \leq
\phi(E(x^*x)(I-p)) = \mu([0,t)) < \eps^2.
\]
\end{proof}

\begin{corollary}\label{sotdensecor} Given $\eps > 0$ and $\zeta \in
\fH_\phi$, there exists $d \in \fn_\phi \cap \D$ and $y \in
\spn\N(\M,\D)$ such that
\[ \norm{\zeta - \eta_\phi(yd)} < \eps.
\]
\end{corollary}

\begin{proof} Since $\eta_\phi(\fn_\phi)$ is dense in $\fH_\phi$, we
may find $x \in \fn_\phi$ such that $\norm{\zeta - \eta_\phi(x)} <
\eps/3$.  By Lemma~\ref{sotdense}, there exists $d\in\fn_\phi\cap \D$
such that $0 \leq d \leq I$ and $\norm{\eta_\phi(x) - \eta_\phi(xd)} <
\eps/3$.  Let $\M_0:=\spn\N(\M,\D)$.  Then $\M_0$ is a unital
$*$-algebra which is $\sigma$-strongly dense in $\M$.  Thus we may
find $y \in \spn\N(\M,\D)$ such that
\[ \norm{\eta_\phi(xd) - \eta_\phi(yd)} = \sqrt{\langle \pi_\phi((x -
y)^*(x - y))\eta_\phi(d), \eta_\phi(d) \rangle} < \eps/3.
\] It follows that $\norm{\zeta - \eta_\phi(yd)} < \eps$.
\end{proof}

Since $\phi\circ E=\phi$,  $\fn_\phi$ and $\fn_\phi^*$ are
$\D$-bimodules and furthermore, for  $D \in \D$, $x \in \fn_\phi$ and
$y\in\fn_\phi^*$ ,
\begin{align}\max\{\phi((Dx)^*(Dx)),\phi((xD)^*(xD))\}&\leq
\norm{D}^2\phi(x^*x)\dstext{and}\label{nphi}\\
\max\{\phi((Dy^*)^*(Dy^*)),\phi((y^*D)^*(y^*D))\}&\leq
\norm{D}^2\phi(yy^*).\label{nphistar}
\end{align}
In particular, for $D\in\D$, the maps on $\eta_\phi(\fn_\phi)$
given by
\[ \pil(D)\eta_\phi(x) = \eta_\phi(Dx) \dstext{and}
\pir(D)\eta_\phi(x) = \eta_\phi(xD)
\] extend to bounded operators $\pil(D)$ and $\pir(D)$ on $\fH_\phi$.
This produces $*$-representations $\pil$ and $\pir$ of $\D$ on
$\fH_\phi$.  Clearly, $$\pil=\pi_\phi|_\D.$$ The relationship
between $\pil$ and $\pir$ is given by Lemma~\ref{relation} below,
whose proof is joint work with Adam Fuller.  The image of $\M$ under
$\pi_\phi$ acts on $\fH_\phi$ in standard form, and we write $J$ for
the modular conjugation operator.

\begin{lemma}\label{relation}
For each $D\in \D$,
$$J\pil(D)J=\pir(D^*).$$
\end{lemma}
\begin{proof}
Throughout the proof, we will freely use notation
from~\cite{TakesakiThOpAlII}, sometimes without explicit mention.

Let $\fA_\phi$ be the full left Hilbert algebra
$\eta_\phi(\fn_\phi\cap\fn_\phi^*)$ (see
\cite[Theorem~VII.2.6]{TakesakiThOpAlII}).  For $x\in
\fn_\phi\cap\fn_\phi^*$ and $D\in \D$,
\begin{equation}\label{firstsharp}
\pil(D)(\eta_\phi(x)^\sharp)=\eta_\phi(Dx^*)=\eta_\phi(xD^*)^\sharp=
(\pir(D^*) \eta_\phi(x))^\sharp.
\end{equation}
The estimates~\eqref{nphi} and~\eqref{nphistar}
 combined with \cite[Lemma~VI.1.4]{TakesakiThOpAlII}
yield that $\fD^\sharp$ is invariant under $\pil(D)$ and
$\pir(D^*)$.  Thus, \eqref{firstsharp} implies that for $\xi\in\fD^\sharp$,
$\pil(D)S\xi =S\pir(D^*)\xi;$ similarly, $S\pil(D)\xi=\pir(D^*)S\xi$.
Hence  \begin{equation}\label{spi}
\pil(D)S=S\pir(D^*)\dstext{and} S\pil(D)=\pir(D^*)S.
\end{equation}
Since $\fD^\flat=\{\zeta\in\fH_\phi:\fD^\sharp\ni\xi\mapsto
\innerprod{\zeta, S\xi} \text{ is bounded}\}$, we see that $\fD^\flat$
is also invariant under $\pil(D^*)$ and $\pir(D)$.
Next, \cite[Lemma~VI.1.5(ii)]{TakesakiThOpAlII} yields,
\begin{equation}\label{fpi}
F\pil(D^*)=\pir(D)F\dstext{and} \pil(D^*)F=F\pir(D).
\end{equation}
Therefore,
$$\Delta\pil(D)=FS\pil(D)=F\pir(D^*)S=\pil(D)FS=\pil(D)\Delta.$$
We thus obtain,
$$\Delta^{1/2}\pil(D)=\pil(D)\Delta^{1/2}.$$
By \cite[Lemma~VI.1.5(v)]{TakesakiThOpAlII}, for
$\xi\in\fD(\Delta^{1/2})=\fD^\sharp$,
$$\pir(D^*)\xi=S\pil(D)S\xi=J\Delta^{1/2}\pil(D)\Delta^{-1/2}J\xi=J\pil(D)J\xi.$$
  Since $\fD^\sharp$ is dense
in $\fH_\phi$ and $\{\pir(D^*),
J\pil(D)J\}\subseteq\B(\fH_\phi)$, the lemma follows.

\end{proof}

\begin{remark}{Notation}
Let
\[ \Z := (\pil(\D) \cup \pir(\D))''.
\] 
\end{remark}
Our first task is to show that $\Z$ is a MASA in $\B(\fH_\phi)$.
While this is established in~\cite[Theorem~1 and
Proposition~2.9(1)]{FeldmanMooreErEqReII}, we provide an alternate
proof (also see~\cite{PopaNoCaSuTyIIFa}).  Our proof has the advantage
that it avoids some of the measure-theoretic issues of the
Feldman-Moore approach, and does not require the separability of
$\M_*$.  

\begin{remark}{Notation} Denote by $P$ the projection on $\fH_\phi$
determined by extending the map $\eta_\phi(\fn_\phi) \ni \eta_\phi(x)
\mapsto \eta_\phi(E(x))$ by continuity.  A calculation shows that for
any $D\in\D$,
\begin{equation}\label{Pcomm} \pil(D)P = \pir(D)P = P\pir(D) =
P\pil(D).
\end{equation}
\end{remark}

\begin{lemma}\label{inZ} For $v \in \G\N(\M,\D)$, set
\[ P_v = \mean_{U\in\U(\D)} \pil(vUv^*)\pir(U^*) \in \B(\fH_\phi).
\] Then $P_v \in \Z$ and the following statements hold.
\begin{enumerate}
\item $P_v = \pi_\phi(v)P\pi_\phi(v)^*$.
\item $P_v$ is the orthogonal projection onto
$\overline{\{\eta_\phi(vd): d \in \fn_\phi \cap \D\}}$, and for $x \in
\fn_\phi$,
\begin{equation}\label{formP} P_v\eta_\phi(x) = \eta_\phi(vE(v^*x)).
\end{equation}
\item If $\xi \in \ran(P_v)$, then there exists $h \in \G\N(\M,\D)$
such that $P_h$ is the projection onto $\overline{\Z\xi}$ and $P_h
\leq P_v$.
\item If $v, w \in \G\N(\M,\D)$, then $P_v \perp P_w$ if and only if
$E(v^*w) = 0$.
\end{enumerate}
\end{lemma}

\begin{proof} Since $v \in \N(\M,\D)$, we have $vUv^* \in \D$ for
every $U \in \U(\D)$.  Hence the function $f(U) =
\pil(vUv^*)\pir(U^*)$ maps $\U(\D)$ into $\Z$, so Lemma~\ref{invmean}
shows that $P_v\in \Z$.

Let $d \in \mathfrak{n}_\phi \cap \D$ satisfy $0 \leq d \leq I$. For $x, y \in \mathfrak{n}_\phi$,
\begin{eqnarray*}
  \langle P_v\eta_\phi(x), \eta_\phi(yd) \rangle
  &=& \mean_{U \in \U(\D)} \langle \pi_\ell(vUv^*)\pi_r(U^*)\eta_\phi(x), \eta_\phi(yd) \rangle\\
  &=& \mean_{U \in \U(\D)} \langle \eta_\phi(vUv^*xU^*), \pi_r(d)\eta_\phi(y) \rangle\\
  &=& \mean_{U \in \U(\D)} \langle \pi_r(d)\eta_\phi(vUv^*xU^*), \eta_\phi(y) \rangle\\
  &=& \mean_{U \in \U(\D)} \langle \eta_\phi(vUv^*xU^*d), \eta_\phi(y) \rangle\\
  &=& \mean_{U \in \U(\D)} \langle \pi_\phi(vUv^*xU^*)\eta_\phi(d), \eta_\phi(y) \rangle\\
  &=& \langle \pi_\phi(vE(v^*x))\eta_\phi(d), \eta_\phi(y) \rangle\\
  &=& \langle \eta_\phi(vE(v^*x)d), \eta_\phi(y) \rangle\\
  &=& \langle \pi_r(d)\eta_\phi(vE(v^*x)), \eta_\phi(y) \rangle\\
  &=& \langle \eta_\phi(vE(v^*x)), \pi_r(d)\eta_\phi(y) \rangle\\
  &=& \langle \eta_\phi(vE(v^*x)), \eta_\phi(yd) \rangle.
\end{eqnarray*}
The equality (2) of part (b) now follows from Lemma 1.4.1. The remainder of part (b) follows from equation (2), which in turn implies (a).

Turning now to the proof of (c), suppose that $\xi \in \ran(P_v)$.
Then $\xi = \pi_\phi(v)\zeta$ for some $\zeta \in \ran(P)$.  For $d_1
\in \D$ and $d_2 \in \fn_\phi \cap \D$, we have that
\[ \pil(d_1)\pi_\phi(v)\eta_\phi(d_2) = \eta_\phi(d_1vd_2) =
\eta_\phi(vd_2v^*d_1v) = \pir(v^*d_1v)\pi_\phi(v)\eta_\phi(d_2).
\] Since $\eta(\fn_\phi \cap \D)$ is dense in $\ran(P)$, it follows
that
\[ \pil(d_1)\xi = \pil(d_1)\pi_\phi(v)\zeta =
\pir(v^*d_1v)\pi_\phi(v)\zeta = \pir(v^*d_1v)\xi,
\] and so $\pil(\D)\xi \subseteq \pir(\D)\xi$. Likewise $\pir(d_1)\xi
= \pil(vd_1v^*)\xi$, and so $\pir(\D)\xi \subseteq \pil(\D)\xi$.  The
fact that $\Z$ is generated by $\pil(\D)$ and $\pir(\D)$ yields
\[ \overline{\pil(\D)\xi} = \overline{\pir(\D)\xi} = \overline{\Z\xi}.
\]

We claim that $\pir(\D)|_{\ran(P_v)}$ is a MASA in $\B(\ran(P_v))$.
Indeed, $\pi_\ell(\D)|_{\ran(P)}$ is a MASA in $\B(\ran(P))$, since
$\pil(\cdot)|_{\ran(P)}$ is unitarily equivalent to $\pi_\omega$, the
semi-cyclic representation of $\D$ corresponding to $\omega := \phi|_{\D}$.
(The implementing unitary $U:\ran(P) \to \fH_\omega$ maps
$\eta_\phi(d)$ to $\eta_\omega(d)$ for all $d \in \fn_\phi \cap \D$.)
It follows that $\pil(\D)|_{\pil(v^*v)\ran(P)}$ is a MASA in
$\B(\pil(v^*v)\ran(P))$.  Now $\pir(\cdot)|_{\ran(P_v)}$ is unitarily
equivalent to $\pil(\cdot)|_{\pil(v^*v)\ran(P)}$.  (The implementing
unitary $V:\ran(P_v) \to \pil(v^*v)\ran(P)$ maps $\eta_\phi(vd)$ to
$\pil(v^*v)\eta_\phi(d)$ for all $d \in \fn_\phi \cap \D$.)  This
establishes the claim.

Now let $Q \in \B(\ran(P_v))$ be the orthogonal projection onto
$\overline{\pir(\D)\xi}$.  Then $Q \in (\pi_r(\D)|_{\ran(P_v)})' =
\pi_r(\D)|_{\ran(P_v)}$, and so there exists a projection $q \in \D$
such that $Q = \pi_r(q)|_{\ran(P_v)}$.  Define $h = vq$.  Then $h \in
\G\N(\M,\D)$, and we have
\begin{align*} \ran(P_h) &= \pi_\phi(h)\ran(P) = \pi_\phi(vq)\ran(P) =
\pi_\phi(v)\pil(q)\ran(P) = \pi_\phi(v)\pir(q)\ran(P)\\ &=
\pir(q)\pi_\phi(v)\ran(P) = \pir(q)\ran(P_v) = \ran(Q) =
\overline{\pir(\D)\xi} = \overline{\Z\xi}.
\end{align*} The fact that $P_h \leq P_v$ follows from the fact that
both are projections and $\ran(P_h)\subseteq \ran(P_v)$.

Finally we prove (d).  For $v, w \in \G\N(\M,\D)$ and $d_1, d_2 \in
\fn_\phi \cap \D$, we have that
\begin{align*} \langle \eta_\phi(vd_1), \eta_\phi(wd_2) \rangle &=
\phi(d_2^*w^*vd_1) = \phi(E(d_2^*w^*vd_1)) = \phi(d_2^*E(w^*v)d_1)\\
&= \omega(d_2^*E(w^*v)d_1) = \langle
\pi_\omega(E(w^*v))\eta_\omega(d_1), \eta_\omega(d_2) \rangle,
\end{align*} and so $P_v \perp P_w$ if and only if $E(w^*v) = 0$.
\end{proof}

\begin{theorem}\label{masa} The algebra $\Z$ is a MASA in
$\B(\fH_\phi)$.
\end{theorem}

\begin{proof} Let $0 \neq Q \in \Z'$ be a projection.  We first show 
there exists $0 \neq h \in
\G\N(\M,\D)$ so that $P_h \leq Q$.

Let $\zeta$ be a unit vector in the range of $Q$.
Corollary~\ref{sotdensecor} implies that there exists $w \in
\N(\M,\D)$ and $d \in \fn_\phi \cap \D$ so that
$\innerprod{\zeta,\eta_\phi(wd)} \neq 0$.  Writing the polar
decomposition, $w = v|w|$, we have $\eta_\phi(wd) =
\pi_\phi(v)\eta_\phi(|w|d) \in \ran(P_v)$. Hence $P_v\zeta\neq 0$.  By
Lemma~\ref{inZ}, $\overline{\Z P_v\zeta}$ is the range of $P_h$ for
some $h \in \G\N(\M,\D)$, and as $\ran(Q)$ is invariant
for $\Z$, we have $P_h \leq Q$.

As $P_h \in \Z \subseteq \Z'$, $Q - P_h \in \Z'$.  A Zorn's Lemma
argument now yields a maximal family $A \subseteq \G\N(\M,\D)$ 
such that (a) $\{P_v: v \in A\}$ is a pairwise
orthogonal family of projections; and (b) $P_v \leq Q$ for each $v\in
A$.  The maximality of $A$ ensures that $\join_{v \in A} P_v = Q$.  As
each $P_v \in \Z$, we conclude that $Q \in \Z$ as well.  Therefore
$\Z$ is a MASA.
\end{proof}

The following extends part of
\cite[Proposition~2.8]{FeldmanMooreErEqReII} to our context.
\begin{corollary}\label{modauto}  Let $\Delta$ be the modular
  operator and $\{\sigma_t^\phi\}_{t\in\bbR}$ be the modular
  automorphism group .  Then for each $t\in\bbR$, 
  $\Delta^{it}\in\U(\Z)$.  Moreover, $\sigma_t^\phi |_\D=\text{id}|_\D$ 
and  for $v\in\G\N(\M,\D)$, $h:=v^*\sigma_t^\phi(v)$ is a partial
isometry in $\D$ and 
  $\sigma_t^\phi(v)=vh$.
\end{corollary}
\begin{proof}
The proof of Lemma~\ref{relation} shows that $\Delta$ commutes with
each element of $\pil(\D)$,  hence for each $t\in\bbR$,
$\Delta^{it}\in\pil(\D)'$.  
Since $J\Delta J=\Delta^{-1}$ (\cite[LemmaVI.1.5(v)]{TakesakiThOpAlII}),
Lemma~\ref{relation} implies that $\Delta^{it}\in \pir(\D)'$.  
 Hence $\Delta^{it}\in\Z'=\Z$.

 For $D\in\D$,
 $\pi_\phi(\sigma_t^\phi(D))=\Delta^{it}\pil(D)\Delta^{-it}=\pi_\phi(D)$, so
 $\sigma_t^\phi$ fixes each element of $\D$.  Let $v\in \G\N(\M,\D)$
 and fix $t\in\bbR$.  Set $w=\sigma_t^\phi(v)$.  We show that
 $v^*w\in\D$ and that $w=v(v^*w)\in v\D$.  To see this, observe that
 for $d\in\D$ we have,
\[wdw^*=\sigma_t^\phi(vdv^*)=vdv^*.\] Therefore for $d\in\D$,
\[v^*wd=v^*(wdw^*)w=v^*(vdv^*)w=dv^*w.\] Since $\D$ is a MASA in $\M$,
$v^*w\in\D$.  Finally, $w=(ww^*)w=v(v^*w)$, as desired.
\end{proof}

We now turn to showing that $\D$ norms $\M$.  We need some general
preparation.  Recall that if $\C\subseteq \bh$ is a \cstar-algebra of
operators, then $\C$ is \textit{locally cyclic} if, for any $\eps>0$,
$n\in\bbN$, and vectors $\xi_1,\dots,\xi_n \in\H$, there is a vector
$\zeta\in \H$ and elements $T_1,\dots, T_n\in\C$ such that for $1\leq
i\leq n$, $\norm{T_i\zeta-\xi_i}<\eps$.

In our context, $\pi_\phi(\M)$ is locally cyclic.   Indeed, we may find
$x_i\in\fn_\phi$ with $\norm{\eta_\phi(x_i)-\xi_i}<\eps/2$. 
Lemma~\ref{sotdense} yields $d\in\D\cap\fn_\phi$ with
$\norm{\eta_\phi(x_i)-\eta_\phi(x_id)}<\eps/2$ for $1\leq i\leq n$;
then $\norm{\pi_\phi(x_i)\eta_\phi(d)-\xi_i}<\eps$.\footnote{A similar
  argument can be used to show that whenever a von Neumann algebra is
  in standard form, it is locally cyclic; we do not need that fact
  here.}  Also, when $\C\subseteq \bh$ is a MASA, $\C$ is locally
cyclic.  This can be proved directly, or one can argue as follows.
Decompose $\H$ into an orthogonal sum of cyclic subspaces,
$\H=\bigoplus_{i\in I} \overline{\C u_i}$ where $\{u_i\}_{i\in
  I}\subseteq \H$ is a family of unit vectors.  As in the proof of
\cite[Theorem~VII.2.7]{TakesakiThOpAlII}, define a faithful normal
semi-finite weight $\phi$ on the positive cone of $\C$ by
$\phi(T)=\sup\{\sum_{i\in F} \innerprod{Tu_i,u_i}: F\subseteq I \text{
  is finite}\}$.  Then the identity representation of $\C$ is
unitarily equivalent to the semi-cyclic representation $(\pi_\phi,
\H_\phi, \eta_\phi)$ and hence $\C$ is locally cyclic because
$(\C,\C)$ is a Cartan pair.

The following is the analog of~\cite[Lemma~2.15]{PittsNoAlAuCoBoIsOpAl} for Cartan pairs. 

\begin{corollary}\label{norming} If $(\M,\D)$ is a Cartan pair, then $\D$ norms $\M$ in the sense of
Pop-Sinclair-Smith~\cite{PopSinclairSmithNoC*Al}.
\end{corollary}
\begin{proof} The proof is an adaptation of the proof of \cite[Proposition~4.1]{SinclairSmithHoCoVNAlCaSu}, with the algebras $\M$, $\A$ and $\B$ of 
\cite[Proposition~4.1]{SinclairSmithHoCoVNAlCaSu} taken to be $\pi_\phi(\M)$, $\pil(\D)$ and $\pir(\D)$ respectively.

Since $\Z$ is a MASA in $\B(\fH_\phi)$, it norms $\B(\H_\phi)$ by \cite[Theorem~2.7]{PopSinclairSmithNoC*Al}.  Then 
$C^*(\A,\B)$ norms $\B(\H_\phi)$ (\cite[Lemma~2.3(c)]{PopSinclairSmithNoC*Al}).  

Let $X\in M_n(\pi_\phi(\M))$ satisfy $\norm{X}=1$ and let $\eps>0$.  
Then there exist $C_1, C_2\in M_{n,1}(C^*(\A,\B))$ such that 
\begin{equation}\label{pss1}
\max\{\norm{C_1},\norm{C_2}\}<1\dstext{and}\norm{C_2^*XC_1}>1-\eps.
\end{equation}
The proof now continues exactly as in the proof of
\cite[Proposition~4.1]{SinclairSmithHoCoVNAlCaSu}: replace the
inequality~(4.2) of~\cite{SinclairSmithHoCoVNAlCaSu} with \eqref{pss1}
and continue the Sinclair-Smith argument from there to show that
$\pil(\D)=\pi_\phi(\D)$ norms $\pi_\phi(\M)$.

\end{proof}

\section{A Spectral Theorem for Bimodules}

In this section, we provide a description of the support of a
$\D$-bimodule in terms of a projection in $\Z$, then use this to
characterize $\D$-bimodules closed in an appropriate topology.

\subsection{The Support of a Bimodule}

\begin{definition}\label{support} For any set $A \subseteq \M$, let
$\innerprod{A}$ be the $\D$-bimodule generated by $A$.
\begin{enumerate}
\item Given any $\D$-bimodule (not necessarily closed) $\S \subseteq
\M$, let
\[ \supp(\S) \in \B(\fH_\phi) \text{ be the orthogonal projection onto
} \overline{\pi_\phi(\S)\eta_\phi(\fn_\phi \cap \D)}.
\] Since $\overline{\pi_\phi(\S)\eta_\phi(\fn_\phi\cap\D)}$ is a
$\Z$-invariant subspace, $\supp(\S)$ is a projection in $\Z$.
\item For $T\in \M$, we define the \textit{support of $T$},
$\supp(T)$, to be the projection $\supp(\innerprod{T}) \in \Z$.
\item Given a projection $Q \in \Z$, the set
\[ \bimod(Q) := \{T \in \M: \supp(T) \leq Q\}
\] is a $\D$-bimodule.
\end{enumerate}
\end{definition}

\begin{remark}{Remark} The purpose of this remark is to outline the
relationship between the notion of support of a bimodule given above
with the notion of support of a bimodule found
in~\cite{MuhlySaitoSolelCoTrOpAl}.  For this, assume that $\M_*$ is
separable, that $\phi$ is a
faithful normal state on $\M$ and use the notation found
in~\cite{FeldmanMooreErEqReII}. By~\cite[Theorem~1]{FeldmanMooreErEqReII},
there exists a countable, standard equivalence relation $R$ on a
finite measure space $(X,\B,\mu)$, a cocycle $\sigma \in H^2(R,\bbT)$,
and an isomorphism of $\M$ onto $\MM(R,\sigma)$ which carries $\D$
onto the diagonal subalgebra $\AA(R,\sigma)$ of
$\mathbf{M}(R,\sigma)$.  We may therefore assume that $\M =
\MM(R,\sigma)$ and that $\D = \AA(R,\sigma)$.  With this
identification, $\M$ acts on the separable Hilbert space $L^2(R,\nu)$,
where $\nu$ is the right counting measure associated with $\mu$.
By~\cite[Proposition~2.9]{FeldmanMooreErEqReII}, $J\D J$ is an abelian
subalgebra of $\M'$ and $\Z = (J\D J \vee \D)''$ is a MASA in
$\B(L^2(R,\nu))$, with cyclic vector $\chi_{\Delta}$ (here $\Delta =
\{(x,x): x \in X\} \subseteq R$).  Each element $a \in \MM(R,\sigma)$
determines a measurable function $a\chi_\Delta$ on $R$, and the
support of such a function is a measurable subset of $R$ determined
uniquely up to null sets.  Projections in $\Z$ are in one-to-one
correspondence with $\nu$-measurable subsets of $R$ modulo null sets,
so we may as well regard the support of an element of $\MM(R,\sigma)$
as a projection in $\Z$.  The support of the $\D$-bimodule $\S$ is the
join of the support projections of the elements of $\S$.  Thus,
Definition~\ref{support} is a reformulation of the definition of the
support of a $\D$-bimodule from~\cite{MuhlySaitoSolelCoTrOpAl}, but
with the measure-theoretic considerations suppressed.
\end{remark}

The following observations will be used in the sequel.

\begin{lemma}\label{GNsupp} Let $h \in \G\N(\M,\D)$. Then $\supp(h) =
P_h$.
\end{lemma}

\begin{proof} Clearly
\[ \ran(P_h) = \overline{\pi_\phi(h)(\fn_\phi \cap \D)} \subseteq
\overline{\pi_\phi(\innerprod{h})(\fn_\phi \cap \D))},
\] and so $P_h \leq \supp(h)$. Conversely, since $\innerprod{h} =
\{hd: d \in \D\}$,
\[ \overline{\pi_\phi(\innerprod{h})(\fn_\phi \cap \D))} \subseteq
\overline{\pi_\phi(h)(\fn_\phi \cap \D)} = \ran(P_h),
\] and so $\supp(h) \leq P_h$.
\end{proof}

\begin{lemma}\label{supch} Let $Q \in \Z$ be a projection.  For
  $T\in\M$,
 the
following are equivalent:
\begin{enumerate}
\item $T \in \bimod(Q)$;
\item $\pi_\phi(T)\eta_\phi(\fn_\phi \cap \D) \subseteq \ran(Q)$;
\item $Q^\perp\pi_\phi(T)P=0$.
\end{enumerate} In particular, if $h \in \G\N(\M,\D)$, then $h \in
\bimod(Q)$ if and only if $P_h \leq Q$.
\end{lemma}

\begin{proof} As the equivalence of (b) and (c) is clear, we show only
  the equivalence of (a) and (b).  
Suppose $T \in \bimod(Q)$.  Then
$\pi_\phi(\innerprod{T})\eta_\phi(\fn_\phi \cap \D) \subseteq
\ran(Q)$, and (b) holds as $T \in \innerprod{T}$.

Conversely, if (b) holds, then for any $h, k \in \D$ and $d \in
\fn_\phi \cap \D$, we have
\[ \pi_\phi(hTk)\eta_\phi(d) = \pil(h)\pir(k)\pi_\phi(T)\eta_\phi(d)
\in \ran(Q)
\] because $\ran(Q)$ is $\Z$-invariant.  So
$\overline{\pi_\phi(\innerprod{T})\eta_\phi(\fn_\phi \cap \D)}
\subseteq \ran(Q)$; hence $T \in \bimod(Q)$.
\end{proof}

The Spectral Theorem for Bimodules from \cite{MuhlySaitoSolelCoTrOpAl}
may be reformulated as the following conjecture.

\begin{flexstate}{Conjecture}{Spectral Theorem for
Bimodules}\label{STB} If $\S$ is a $\sigma$-weakly closed
$\D$-bimodule in $\M$, then $\S = \bimod(\supp(\S))$, that is,
\begin{equation}\label{STBinc} \S = \{T \in \M:
\pi_\phi(T)\eta_\phi(\fn_\phi \cap \D) \subseteq
\overline{\pi_\phi(\S)\eta_\phi(\fn_\phi \cap \D)}\}.
\end{equation}
\end{flexstate}

\begin{remark}{Remarks} For these remarks, assume $\phi$ is a faithful
normal state, so that $\eta_\phi(I)$ is a cyclic and separating vector
for $\pi_\phi(\M)$.
\begin{enumerate}
\item Observe that replacing $\eta_\phi(\fn_\phi \cap \D)$ with
$\eta_\phi(I)$ in Definition~\ref{support} leaves the definition of
$\supp(\S)$ unchanged; this replacement may also be made in
\eqref{STBinc}.  Thus, the Spectral Theorem for Bimodules is the same
as the equality
\[ \S = \{T \in \M: \pi_\phi(T)\eta_\phi(I) \in
\overline{\pi_\phi(\S)\eta_\phi(I)}\}.
\]
\item What is known (see~\cite[Theorem~2.3]{LoginovSulmanHeInReW*Al})
is that when $\S$ is a $\sigma$-weakly closed subspace of $\M$, then
because $\pi_\phi(\M)$ has a separating vector, $\pi_\phi(\S)$ is
reflexive, that is,
\[ \pi_\phi(\S) = \{T \in \B(\fH_\phi): T\xi \in
\overline{\pi_\phi(\S)\xi} \text{ for every } \xi \in \fH_\phi\}.
\] The faithfulness of $\pi_\phi$ then yields
    \[ \S=\{T\in\M: \pi_\phi(T)\xi\in \overline{\pi_\phi(\S)\xi}
\text{ for every $\xi\in \fH_\phi$}\}.
\] Clearly,
\begin{align*} \S &= \{T \in \M: \pi_\phi(T)\xi \in
\overline{\pi_\phi(\S)\xi} \text{ for all } \xi \in \fH_\phi\}\\
&\subseteq \{T \in \M: \pi_\phi(T)\eta_\phi(I) \in
\overline{\pi_\phi(\S)\eta_\phi(I)}\}.
\end{align*} Thus, Conjecture~\ref{STB} holds if and only if the
inclusion is an equality.  (This is roughly the approach attempted in
\cite{MuhlySaitoSolelCoTrOpAl}.)
\item Since $\eta_\phi(I)$ is a cyclic and separating vector for
$\pi_\phi(\M)$, it is also cyclic and separating for $\pi_\phi(\M)'$.
If $T \in \M$ and $\pi_\phi(T)\eta_\phi(I) \in
\overline{\pi_\phi(\S)\eta_\phi(I)}$, then for each $Y\in
\pi_\phi(\M)'$, we have
\[ \pi_\phi(T)Y\eta_\phi(I) = Y\pi_\phi(T)\eta_\phi(I) \in
Y\overline{\pi_\phi(\S)\eta_\phi(I)} \subseteq
\overline{\pi_\phi(\S)Y\eta_\phi(I)}.
\] Hence
\[ \{T \in \M: \pi_\phi(T)\eta_\phi(I) \in
\overline{\pi_\phi(\S)\eta_\phi(I)}\} = \{T \in \M: \pi_\phi(T)\xi\in
\overline{\pi_\phi(\S)\xi} \text{ for all } \xi \in
\pi_\phi(\M)'\eta_\phi(I)\}.
\] Thus we see that Conjecture~\ref{STB} holds if and only if the
inclusion
\[ \{T \in \M: \pi_\phi(T)\xi \in \overline{\pi_\phi(\S)\xi} \text{
for all } \xi \in \fH_\phi\} \subseteq \{T \in \M: \pi_\phi(T)\xi\in
\overline{\pi_\phi(\S)\xi} \text{ for all } \xi \in
\pi_\phi(\M)'\eta_\phi(I)\}
\] is actually an equality.
\end{enumerate}
\end{remark}

\subsection{Topologies}

In this subsection we discuss the Bures and $L^2$ topologies on
$\M$. We begin with a fact well-known to experts in
non-commutative integration.
\begin{lemma}\label{vecFunct} Let $\M$ be a von Neumann algebra, let
$\phi$ be a faithful, semi-finite, normal weight on $\M$, and let
$(\pi_\phi,\fH_\phi,\eta_\phi)$ be the semi-cyclic representation of
$\M$ arising from $\phi$.  If $f \in \M_*$, then there are vectors
$\xi_1, \xi_2 \in \fH_\phi$ such that for every $x \in \M$, $f(x) =
\innerprod{\pi_\phi(x)\xi_1,\xi_2}$.
\end{lemma}

\begin{proof} By the polar decomposition for normal functionals on a
von Neumann algebra (\cite[Theorem~III.4.2(i)]{TakesakiThOpAlI}),
there exists a partial isometry $v \in \M$ and $\rho \in (\M_*)^+$
such that for each $x \in \M$,
\begin{equation}\label{vecFunct1} f(x)=\rho(xv).
\end{equation} Since $\pi_\phi$ puts $\M$ into standard form,
\cite[Theorem~IX.1.2]{TakesakiThOpAlII} shows there exists $\xi_2 \in
\fH_\phi$ such that for every $x \in \M$, $\rho(x) =
\innerprod{\pi_\phi(x)\xi_2,\xi_2}$.  Taking $\xi_1 :=
\pi_\phi(v)\xi_2$, the lemma follows from \eqref{vecFunct1}.
\end{proof}

As noted earlier, the semi-cyclic representation of $\D$ induced by
$\phi|_\D$ is unitarily equivalent to $(\pil,
\ran(P),\eta_\phi|_{\fn_\phi \cap \D})$.  Thus, given $f \in \D_*$
there are $\xi_1, \xi_2 \in \ran(P)$ so that for every $D \in \D$,
\[ f(D) = \innerprod{\pil(D)\xi_1,\xi_2}.
\]

\begin{lemma}\label{charBsemin} The two families of semi-norms on
$\M$,
\[ \{\M \ni T \mapsto \sqrt{\tau(E(T^*T))}: \tau \in (\D_*)^+\}
\dstext{and} \{\M \ni T \mapsto \norm{\pi_\phi(T)\xi}: \xi \in
\ran(P)\},
\] coincide.
\end{lemma}

\begin{proof} Given $\tau \in (\D_*)^+$, there exists $\xi \in
\ran(P)$ so that $\tau(d) = \innerprod{\pil(d)\xi,\xi}$.  Choose $h_n
\in \fn_\phi \cap \D$ so that $\eta_\phi(h_n) \rightarrow \xi$.  Then
\begin{align*} \tau(E(T^*T)) &= \innerprod{\pil(E(T^*T))\xi,\xi} =
\lim_{n \to \infty}
\innerprod{\pil(E(T^*T))\eta_\phi(h_n),\eta_\phi(h_n)}\\ &=
\lim_{n\rightarrow\infty} \norm{\pi_\phi(T)\eta_\phi(h_n)}^2 =
\norm{\pi_\phi(T)\xi}^2.
\end{align*} It follows that $\{\M \ni T \mapsto \sqrt{\tau(E(T^*T))}:
\tau \in (\D_*)^+\} \subseteq \{\M \ni T \mapsto
\norm{\pi_\phi(T)\xi}: \xi \in \ran(P)\}$.  The reverse inclusion is
left to the reader.
\end{proof}

We require two topologies on $\M$, both discussed
in~\cite{MercerBiOvCaSu}, but the second is extended slightly here.

\begin{definition}
\begin{enumerate}
\item The \textit{Bures} topology (see~\cite[page~48]{BuresAbSuvNAl})
on $\M$ is the locally convex topology generated by the family of
seminorms
\[ \fT_B := \{\M \ni T \mapsto \sqrt{\tau(E(T^*T))}: \tau \in
(\D_*)^+\} = \{\M \ni T \mapsto \norm{\pi_\phi(T)\xi}: \xi \in
\ran(P)\}.
\] We denote the Bures topology by $\tau_B$.
\item The \textit{$L^2$ topology} on $\M$ is the topology on $\M$
induced by the family of seminorms
\[ \{\M \ni T \mapsto \norm{\pi_\phi(T)\eta_\phi(d)}: d \in \fn_\phi
\cap \D\}.
\] We will use $(\M,L^2)$ to denote $\M$ equipped with the $L^2$
topology.
\end{enumerate}
\end{definition}

\begin{remark}{Remark} When $\phi$ is a faithful normal state on $\M$,
the $L^2$ topology is determined by the single seminorm $\M \ni T
\mapsto \norm{\pi_\phi(T)\eta_\phi(I)} = \norm{\eta_\phi(T)}$, and in
this case, the $L^2$ topology was considered by Mercer
in~\cite{MercerBiOvCaSu}.  When $\D$ is isomorphic to
$L^\infty(X,\mu)$, $\fn_\phi \cap \D$ may be thought of as $L^2 \cap
L^\infty$, so it is tempting to use the term ``bounded Bures
topology'' instead of the $L^2$-topology, but we have chosen to stay
with the nomenclature used by Mercer.
\end{remark}

Clearly the $L^2$-topology is coarser than the Bures topology, which
in turn is coarser than the norm topology, so the dual spaces of $\M$
equipped with these topologies satisfy
\[ \dual{(\M,L^2)} \subseteq \dual{(\M,\tau_B)} \subseteq
\dual{(\M,\text{norm})}.
\]

\begin{corollary}\label{vector} For $\xi \in \ran(P)$ and $\zeta \in
\fH_\phi$, the functional $T \mapsto \innerprod{\pi_\phi(T)\xi,\zeta}$
belongs to $\dual{(\M,\tau_B)}$.
\end{corollary}

\begin{proof} By the Cauchy-Schwartz inequality,
$|\innerprod{\pi_\phi(T)\xi,\zeta}| \leq
\norm{\pi_\phi(T)\xi}\norm{\zeta}$.  By Lemma~\ref{charBsemin}, the
first term in the product is one of the seminorms defining the Bures
topology.  The corollary follows.
\end{proof}

We now show that every Bures-continuous linear functional is of this
form.

\begin{lemma}\label{wstcont} Let $f$ be a linear functional on $\M$.
\begin{enumerate}
\item If $f$ is $\tau_B$ continuous, then there exist $\xi \in
\ran(P)$ and $\zeta \in \fH_\phi$ such that
\[ f(T) = \innerprod{\pi_\phi(T)\xi,\zeta}.
\] In particular, $f$ is $\sigma$-weakly continuous on $\M$.
\item If $f \in \dual{(\M,L^2)}$, then there exists $d \in \fn_\phi
\cap \D$ and $\zeta\in \fH_\phi$ such that
\[ f(T) = \innerprod{\pi_\phi(T)\eta_\phi(d),\zeta}.
\]
\end{enumerate} Moreover, $\dual{(\M,\tau_B)}$ and $\dual{(\M,L^2)}$
are norm-dense in $\M_*$.
\end{lemma}

\begin{proof} For the first statement, we give a standard argument
(see the proof of~\cite[Lemma~II.2.4]{TakesakiThOpAlI}).  Since $f$ is
$\tau_B$ continuous, there exist $p_1, \dots, p_n \in \fT_B$ such that
for every $T\in \M$, we have
\[ |f(T)| \leq \sum_{k=1}^n p_k(T)
\] (see \cite[Theorem~IV.3.1]{ConwayCoFuAn}).  Write $p_k(T) =
\sqrt{\omega_k(E(T^*T))}$, where the $\omega_k$ are positive normal
functionals on $\D$.  Set $\omega = n\sum_{k=1}^n \omega_k$ and let
$p(T) = \sqrt{\omega(E(T^*T))}$.  By the Cauchy-Schwartz inequality,
\begin{equation}\label{HBextend} |f(T)| \leq p(T).
\end{equation} By Lemma~\ref{charBsemin}, there is a vector $\xi \in
\ran(P)$ such that for $T \in \M$,
\[ p(T) = \norm{\pi_\phi(T)\xi}.
\] By~\eqref{HBextend}, the map
\[ \pi_\phi(T)\xi \mapsto f(T)
\] is bounded on the subspace $\{\pi_\phi(T)\xi: T \in \M\} \subseteq
\fH_\phi$.  The Riesz Representation Theorem implies that there exists
a vector $\zeta \in \overline{\{\pi_\phi(T)\xi: T \in \M\}} \subseteq
\fH_\phi$ such that
\[ f(T) = \innerprod{\pi_\phi(T)\xi,\zeta}.
\] Hence $f$ is $\sigma$-weakly continuous on $\M$.

The proof of statement (b) is similar and left to the reader.

Suppose $T \in \M$ and $f(T) = 0$ for every $f \in \dual{(\M,L^2)}$.
For every $d \in \fn_\phi \cap \D$, the map $\M \ni S \mapsto
\innerprod{\pi_\phi(S)\eta_\phi(d),\pi_\phi(T)\eta_\phi(d)}$ belongs
to $\dual{(\M,L^2)}$, so
$\innerprod{\pi_\phi(T)\eta_\phi(d),\pi_\phi(T)\eta_\phi(d)} = 0$.
Hence $\innerprod{\pil(E(T^*T))\eta_\phi(d),\eta_\phi(d)} = 0$ for
each $d \in \fn_\phi \cap \D$.  This implies that $E(T^*T) = 0$, and
hence $T = 0$.  It follows that $\dual{(\M,L^2)}$ is weakly dense in
$\M_*$.  As $\dual{(\M,L^2)}$ is a subspace, its weak and norm
closures coincide, so
\[ \M_* = \overline{\dual{(\M,L^2)}}^{\sigma(\M_*,\M)} =
\overline{\dual{(\M,L^2)}}^{\norm{}}.
\] Thus, $\dual{(\M,L^2)}$ is norm dense in $\M_*$.  Since every $L^2$
continuous linear functional is Bures continuous, the Bures continuous
linear functionals are norm dense in $\M_*$ also.
\end{proof}

\begin{corollary}\label{biggerclosure} Let $C$ be a convex set in
$\M$.  Then $\overline{C}^{\sigma\text{-weak}} \subseteq
\overline{C}^{\text{Bures}} \subseteq \overline{C}^{L^2}$, with
equality throughout if $C$ is also a bounded set.
\end{corollary}

\subsection{$\sigma$-Weakly Closed Bimodules}

\begin{lemma}\label{split} Let $\S \subseteq \M$ be a $\sigma$-weakly
closed $\D$-bimodule.  Then the following statements hold.
\begin{enumerate}
\item If $u \in \N(\M,\D)$, there exists a projection $Q \in \D$ such
that $uQ \in \S$ and $uQ^\perp$ satisfies $E((uQ^\perp)^*S) = 0$ for
every $S \in \S$.
\item If $X \in \bimod(\supp(\S))$, then for every $u \in
\N(\M,\D)$, $uE(u^*X) \in \S$.
\item Let $\S_B$ be the Bures closure of $\S$.  Then
$\supp(\S)=\supp(\S_B)$.
\end{enumerate}
\end{lemma}

\begin{proof} Let $u \in \N(\M,\D)$, and set $J := \{d \in \D: ud \in
\S\}$.  Since $\S$ is a bimodule, $J$ is an ideal in $\D$, and the
fact that $\S$ is $\sigma$-weakly closed ensures that $J$ is
also $\sigma$-weakly closed.  Therefore, there exists a unique projection $Q
\in \D$ such that $J =\D Q$.  Obviously, $Q \in J$ and $uQ^\perp \in
\N(\M,\D)$.  Proposition~\ref{weaknormalizerexists} shows that if $S
\in \S$, then $uQ^\perp E((uQ^\perp)^*S) \in \S$.  Thus $uQ^\perp
E((uQ^\perp)^*S) = uQ^\perp E(u^*S) \in \S$, and hence $Q^\perp
E(u^*S) \in J$.  It follows that $0 = Q^\perp E(u^*S) =
E((uQ^\perp)^*S)$, as desired.

Turning now to part (b), suppose first $u \in \G\N(\M,\D)$ and $X \in
\bimod(\supp(\S))$.  Then $\pi_\phi(X)\eta_\phi(\fn_\phi \cap \D)
\subseteq \overline{\pi_\phi(\S)\eta_\phi(\fn_\phi \cap \D)}$.  Let
$Q$ be the projection obtained as in part (a).  For any $S \in \S$ and
$d \in \fn_\phi \cap \D$, using Lemma~\ref{inZ}(b), we have
\[ P_{uQ^\perp}(\pi_\phi(S)\eta_\phi(d)) = \pi_\phi(uQ^\perp
E((uQ^\perp)^*S))\eta_\phi(d) = 0.
\] Hence, for any $S \in \S$ and $h \in \fn_\phi \cap \D$,
\[ \norm{P_{uQ^\perp}(\pi_\phi(X)\eta_\phi(h))}
=\norm{P_{uQ^\perp}(\pi_\phi(X)\eta_\phi(h)) -
P_{uQ^\perp}(\pi_\phi(S)\eta_\phi(d))} \leq
\norm{\pi_\phi(X)\eta_\phi(h) - \pi_\phi(S)\eta_\phi(d)}.
\] Holding $h$ fixed and taking the infimum over $S \in \S$ and $d \in
\fn_\phi \cap \D$, the hypothesis on $X$ gives
\[ 0 = P_{uQ^\perp}(\pi_\phi(X)\eta_\phi(h)) = \pi_\phi(uQ^\perp
E((uQ^\perp)^*X))\eta_\phi(h) = \pi_\phi(uE(u^*X)Q^\perp)\eta_\phi(h).
\] Setting $y := uE(u^*X)Q^\perp$, this shows that for every $h \in
\fn_\phi \cap \D$ we have
\[ 0 = \phi(E(h^*y^*yh)) = \phi(h^*hE(y^*y)).
\] Thus, for every $\tau \in \D_*^+$, $\tau(E(y^*y)) = 0$.  This shows
that $E(y^*y) = 0$, and by faithfulness of $E$, $y = 0$; thus,
$uE(u^*X)Q^\perp = 0$.  Hence
\[ uE(u^*X) = uE(u^*X)Q \in \S.
\]

Now let $u\in\N(\M,\D)$, with $u\neq 0$ (the case when $u=0$ is
trivial).  If $u=w|u|$ is the polar decomposition of $u$,  
Lemma~\ref{polarnormal} gives $w\in\G\N(\M,\D)$.   Then $|u|^2=u^*u\in\D$ and
$wE(w^*X)\in\S$.  Since
$$uE(u^*X)=wE(w^*X)u^*u\in \S,$$  the proof of (b) is complete.

To show that (c) holds, we must show that
$\overline{\pi_\phi(\S)\eta_\phi(\fn_\phi \cap \D)} =
\overline{\pi_\phi(\S_B)\eta_\phi(\fn_\phi \cap \D)}$. Since $\S
\subseteq \S_B$, we obtain $\overline{\pi_\phi(\S)\eta_\phi(\fn_\phi \cap
\D)} \subseteq \overline{\pi_\phi(\S_B)\eta_\phi(\fn_\phi \cap \D)}$.
If $T \in \S_B$, we may find a net $(T_\lambda)$ in $\S$ converging in the 
Bures topology to $T$.  Then $T_\lambda \stackrel{L^2}{\longrightarrow} T$,
and hence given $d \in \fn_\phi \cap \D$,
$\pi_\phi(T_\lambda)\eta_\phi(d)\rightarrow \pi_\phi(T)\eta_\phi(d)$.
Therefore, $\pi_\phi(T)\eta_\phi(d)\in
\overline{\pi_\phi(\S)\eta_\phi(d)}$.  Thus,
$\overline{\pi_\phi(\S_B)\eta_\phi(\fn_\phi \cap \D)} \subseteq
\overline{\pi_\phi(\S)\eta_\phi(\fn_\phi \cap \D)}$ and part (c)
follows.
\end{proof}

\begin{corollary}\label{GNcapS} Let $\S \subseteq \M$ be a
$\sigma$-weakly closed $\D$-bimodule and $h \in \G\N(\M,\D)$.  Then $h
\in \S$ if and only $P_h \leq \supp(\S)$.  Thus $\supp(\S) =
\bigvee_{h \in \S \cap \G\N(\M,\D)} P_h$.
\end{corollary}

\begin{proof} Suppose $h \in \S$. Then $h \in \bimod(\supp(\S))$, and so $P_h \leq \supp(\S)$, by Lemma 2.1.4. Conversely, suppose $P_h \leq \supp(\S)$. Then, again by Lemma 2.1.4, $h \in \bimod(\supp(\S))$. By Lemma 2.3.1(b), $h = hE(h^*h) \in \S$.

By the proof of Theorem~\ref{masa}, $\supp(\S) = \bigvee_{h \in A}
P_h$, for some $A \subseteq \G\N(\M,\D)$.  For $h \in A$, $P_h \leq
\supp(\S)$, and so $h \in \S$. Thus
\[ \supp(\S) = \bigvee_{h \in A} P_h \leq \bigvee_{h \in \S \cap
\G\N(\M,\D)} P_h \leq \supp(\S).
\]
\end{proof}

\subsection{$\D$-Orthogonality}

\begin{definition} A non-empty set $\E \subseteq \G\N(\M,\D)
\backslash \{0\}$ is called \textit{$\D$-orthogonal} if for every
$v_1, v_2 \in \E$ with $v_1 \neq v_2$, $E(v_1^*v_2) = 0$
(equivalently, $P_{v_1} \perp P_{v_2}$, by Lemma~\ref{inZ}(d)).
\end{definition}

A simple Zorn's Lemma argument shows the existence of a maximal
$\D$-orthogonal set.

\begin{remark}{Remark}\label{flip}  Notice that for $v_1, v_2\in\G\N(\M,\D)$,
  $v_1$ and $v_2$ are $\D$-orthogonal if and only if $v_1^*$ and
  $v_2^*$ are $\D$-orthogonal.  Indeed, $E(v_1^*v_2)=0$ implies
  $0=v_1E(v_1^*v_2)v_1^*=E(v_1v_1^*v_2v_1^*)=E(v_2v_1^*)$; the
  converse is similar.
\end{remark}  

\begin{lemma}\label{SOTconv} Let $\E \subseteq \G\N(\M,\D) \backslash
\{0\}$ be a maximal $\D$-orthogonal set. Then $\sum_{u \in \E} P_u =
I$, where the sum converges strongly in $\B(\fH_\phi)$.
\end{lemma}

\begin{proof} Let $Q = \sum_{u \in \E} P_u \in \Z$.  If $I - Q \neq
0$, then by the proof of Theorem~\ref{masa}, there exists $0 \neq h
\in \G\N(\M,\D)$ such that $P_h \leq I - Q$.  Then $P_h \perp P_u$ for
all $u \in \E$, contradicting  maximality of $\E$.
\end{proof}

The following is an adaptation of a result of  Mercer to our
context.  
\begin{proposition}[cf. {\cite[Theorem~4.4]{MercerBiOvCaSu}}]\label{BConv} Let $\E \subseteq \G\N(\M,\D)
\backslash \{0\}$ be a maximal $\D$-orthogonal set and let $\Gamma$ be
the set of all finite subsets of $\E$ directed by inclusion. Fix $X
\in \M$.  For $F \in \Gamma$, let
\[ X_F = \sum_{u\in F} uE(u^*X).
\] Then $(X_F)_{F \in \Gamma}$ is a net which converges in the Bures
topology to $X$.
\end{proposition}

\begin{proof} Let $d \in \fn_\phi \cap \D$.  Lemma~\ref{inZ}(b) gives,
\[ \pi_\phi(X_F)\eta_\phi(d) = \sum_{u \in F} \eta_\phi(uE(u^*Xd)) =
\sum_{u\in F} P_u\eta_\phi(Xd)=\sum_{u\in F}
P_u\pi_\phi(X)\eta_\phi(d),
\] and hence for every $\xi \in \overline{\eta_\phi(\fn_\phi \cap
\D)}$,
\[ \pi_\phi(X_F)\xi = \sum_{u\in F} P_u\pi_\phi(X)\xi.
\] Since $I = \sum_{u\in\E} P_u$ (where the sum converges strongly in
$\B(\fH_\phi)$), for every $\xi \in \overline{\eta_\phi(\fn_\phi \cap
\D)}$,
\[ \pi_\phi(X_F)\xi \rightarrow \pi_\phi(X)\xi.
\] Therefore, $X_F \stackrel{\text{Bures}}{\rightarrow} X$.
\end{proof}

\subsection{A Characterization of Bures Closed Bimodules}

The following is a version of the Spectral Theorem for Bimodules,
which characterizes Bures (or $L^2$) closed bimodules.

\begin{theorem}\label{newSTB} Let $\S \subseteq \M$ be a
$\D$-bimodule.  Then the following statements are equivalent:
\begin{enumerate}
\item $\S = \bimod(\supp(\S))$;
\item $\S$ is $L^2$-closed;
\item $\S$ is Bures-closed;
\item $\S$ is the smallest Bures-closed $\D$-bimodule containing $\S
\cap \G\N(\M,\D)$.
\end{enumerate}
\end{theorem}

\begin{proof} Suppose (a) holds.  Let $(T_\lambda)$ in $\S$ be such
that $T_\lambda \stackrel{L^2}{\longrightarrow} T \in \M$.  Then given
$d \in \fn_\phi \cap \D$, $\pi_\phi(T_\lambda)\eta_\phi(d) \rightarrow
\pi_\phi(T)\eta_\phi(d)$.  Therefore, $\pi_\phi(T)\eta_\phi(d) \in
\overline{\pi_\phi(\S)\eta_\phi(d)}$.  Then Lemma~\ref{supch} gives $T
\in \bimod(\supp(\S)) = \S$.  Thus $\S$ is $L^2$-closed.

As the Bures topology is stronger than the $L^2$-topology, we see that
(b) $\Rightarrow$ (c).

We now establish (c) $\Rightarrow$ (a).  Suppose $X \in
\bimod(\supp(\S))$.  Let $\E \subseteq \G\N(\M,\D) \backslash \{0\}$
be a maximal $\D$-orthogonal set.  By Lemma~\ref{split}(b) and
Proposition~\ref{BConv}, $X_F \in \S$ and $X_F
\stackrel{\text{Bures}}{\rightarrow} X$; hence $X \in
\overline{\S}^{\text{Bures}} = \S$.  Thus, $\bimod(\supp(\S))\subseteq
\S$.  As the reverse inclusion is obvious, (a) holds.

Let $\S_1 = \innerprod{\S \cap \G\N(\M,\D)}_B$, the smallest
Bures-closed $\D$-bimodule containing $\S \cap \G\N(\M,\D)$.  If $\S =
\S_1$, then $\S$ is Bures closed, thus (d) $\Rightarrow$ (c).
Conversely, suppose $\S$ is Bures-closed.  Then $\S_1 \subseteq \S$,
clearly.  On the other hand, $\S \cap \G\N(\M,\D) \subseteq \S_1 \cap
\G\N(\M,\D)$, which implies $\supp(\S) \subseteq \supp(\S_1)$, by
Corollary~\ref{GNcapS}, and so $\S = \bimod(\supp(\S)) \subseteq
\bimod(\supp(\S_1)) = \S_1$, using the equivalence of (a) and (c).
\end{proof}

\begin{corollary}\label{BuresClosure} Let $\S \subseteq \M$ be a
$\sigma$-weakly closed $\D$-bimodule.  Then
$\overline{\S}^{\text{Bures}} = \bimod(\supp(\S))$.
\end{corollary}

\begin{proof} By Theorem~\ref{newSTB} and Lemma~\ref{split}(c),
\[ \overline{\S}^{\text{Bures}} =
\bimod(\supp(\overline{\S}^{\text{Bures}})) = \bimod(\supp(\S)).
\]
\end{proof}

Let $\L$ be a commutative subspace lattice acting on the Hilbert space
$\H$.  Consider the family $\R$ of all $\sigma$-weakly closed
subalgebras $\A$ of $\bh$ such that $\A \cap \A^* = \L'$ and $\lat(\A)
= \L$.  Arveson~\cite[Theorem~2.1.8]{ArvesonOpAlInSu} showed that
relative to set inclusion, the family $\R$ has a minimal element,
$\A_{\text{min}}(\L)$, and $\alg(\L)$ is the maximal element of $\R$.
The following proposition has the same flavor.

\begin{proposition}\label{specsyn} Let $Q \in \Z$ and let $\fB$ be the
set of all $\sigma$-weakly closed $\D$-bimodules $\S \subseteq \M$
such that $\supp(\S) = Q$.  Then $\overline{\innerprod{\bimod(Q) \cap
\G\N(\M,\D)}}^{\sigma\text{-weak}}$ is the minimal element of $\fB$
and $\bimod(Q)$ is the maximal element of $\fB$.
\end{proposition}

\begin{proof} First we show that $\supp(\bimod(Q)) = Q$.  Indeed, for
$h \in \G\N(\M,\D)$,
\[ h \in \bimod(Q)  \iff P_h \leq Q,
\] by Lemma~\ref{supch}.  Therefore, 
\[ \supp(\bimod(Q)) = \bigvee\{P_h: h \in \bimod(Q) \cap \G\N(\M,\D)\}
= \bigvee\{P_h: h \in \G\N(\M,\D), ~ P_h \leq Q\} = Q.
\] Thus $\bimod(Q) \in \fB$.

Now let $\S_0 = \overline{\innerprod{\bimod(Q) \cap
\G\N(\M,\D)}}^{\sigma\text{-weak}}$, a $\sigma$-weakly closed
$\D$-bimodule in $\M$. Then
\[ \bimod(Q) \cap \G\N(\M,\D) \subseteq \S_0 \cap \G\N(\M,\D)
\subseteq \bimod(Q) \cap \G\N(\M,\D).
\] 
Corollary~\ref{GNcapS} gives $\supp(\S_0) = \supp(\bimod(Q)) = Q$, 
which implies $\S_0 \in \fB$.

Finally, if $\S \in \fB$, then for $h \in \G\N(\M,\D)$,
\[ h \in \S \iff P_h \leq \supp(\S) = Q \iff h \in \bimod(Q),
\] by Corollary~\ref{GNcapS} and Lemma~\ref{supch}.   Therefore,
\[ \S_0 = \overline{\innerprod{\bimod(Q) \cap
\G\N(\M,\D)}}^{\sigma\text{-weak}} = \overline{\innerprod{\S \cap
\G\N(\M,\D)}}^{\sigma\text{-weak}} \subseteq \S \subseteq
\bimod(\supp(\S)) = \bimod(Q).
\]
\end{proof}

\begin{remark}{Remark} By Lemma~\ref{polarnormal},
\[ \innerprod{\bimod(Q) \cap \G\N(\M,\D)} = \spn(\bimod(Q) \cap
\N(\M,\D)),
\] and so $\overline{\spn}^{\sigma\text{-weak}}(\bimod(Q) \cap
\N(\M,\D))$ is another expression for the minimal element of $\fB$.
\end{remark}

\begin{remark}{Remark} Using the failure of spectral synthesis for an
appropriate locally compact abelian group, in~\cite{ArvesonOpAlInSu},
Arveson constructed a commutative subspace lattice $\L$ for which
$\A_{\text{min}}(\L) \subsetneq \alg(\L)$.  This, together with
Proposition~\ref{specsyn}, suggests that Conjecture~\ref{STB} may not
hold in general.
\end{remark}

While our context differs from that of~\cite{ArvesonOpAlInSu}, the
parallels are sufficiently strong that we make the following
definition.

\begin{definition}\label{synthetic} Let $\S\subseteq \M$ be a
  $\sigma$-weakly closed $\D$-bimodule.  We say $\S$ is
  \textit{synthetic}, or satisfies \textit{spectral synthesis}, if the
  minimal and maximal $\sigma$-weakly closed bimodules with
  $\supp(\S)$ coincide, that is, if
\[\overline{\spn}^{\sigma\text{-weak}}(\S \cap
\N(\M,\D))=\S=\overline{\S}^{\text{Bures}}.\] 
\end{definition}  

\begin{remark}{Remark}\label{synagree}    Let $\A$ be a CSL algebra.
  We wish to point out that when $\A$ contains a Cartan MASA in $\bh$,
  our notion of synthesis and Arveson's notion coincide.

  Let $\H$ be a separable Hilbert space, let $\{e_n\}$ be an
  orthonormal basis for $\H$ and let $\D$ be the atomic MASA of all
  operators diagonal with respect to this basis.  Then $(\bh,\D)$ is a
  Cartan pair (all Cartan MASAs in $\bh$ are of this form).  Let
  $e_ie_j^*$ denote the rank-one operator $\xi\mapsto
  \innerprod{\xi,e_j}e_i$.  Taking $\phi$ to be the tracial weight on
  $\bh$, $\fH_\phi$ is the set of Hilbert-Schmidt operators.  Each
  minimal projection in $\Z\subseteq \B(\fH_\phi)$ has range $\bbC
  \eta_\phi(e_ie_j^*)$ for some $i,j\in\bbN$, and it follows that
  $\Z\subseteq \B(\fH_\phi)$ is an atomic MASA.
  
  If $\A\subseteq \bh$ is a $\sigma$-weakly closed algebra with
  $\D\subseteq \A$, then for each finite-rank projection $P\in\D$,
  $P\A P$ is spanned by $\{PvP: v\in\A\cap \G\N(\bh,\D)\}$.  Since $I$
  may be written as the strong limit of an increasing sequence of such
  projections, the span of the rank one operators contained in $\A$ is
  $\sigma$-weakly dense in $\A$.  Using the description of the atoms
  of $\Z$ from above, one shows $\A$ is synthetic in the sense of
  Definition~\ref{synthetic}.
  Moreover,~\cite[Theorem~23.7]{DavidsonNeAl} shows $\A$ is a
  completely distributive CSL algebra.  Hence
  by~\cite[Corollary~23.9]{DavidsonNeAl}, $\A$ is synthetic in
  Arveson's sense as well.
\end{remark}

The following consequence of Theorem~\ref{newSTB} and the proof of
Proposition~\ref{specsyn} is
worth noting; we leave the proof to the reader.

\begin{theorem}\label{lattice}
Let  $\fS$ be the
lattice of all Bures-closed $\D$-bimodules of $\M$ (where $\meet$ is
intersection and $\join$ is Bures-closed span) and 
let $\fL$ be the projection lattice of $\Z$.  The maps $\bimod:
\fL\rightarrow \fS$ and $\supp: \fS\rightarrow \fL$ are lattice
isomorphisms and $(\bimod)^{-1}=\supp$.
\end{theorem}

We close this section by showing that the class of von Neumann
subalgebras which lie between $\D$ and $\M$ is a class of
$\sigma$-weakly closed $\D$-bimodules which is well-behaved with
respect to the operations of $\bimod$ and $\supp$.  Suppose $\A$ is a
von Neumann algebra with $\D\subseteq \A\subseteq \M$.  Then 
$\A_0:=\overline{\spn}^{\sigma\text{-weak}}(\G\N(\M,\D)\cap\A)$ is
also a von Neumann algebra.  A much less obvious fact, contained in
Theorem~\ref{aoi} below, is that the Bures-closure,
$\A_B:=\overline{\A}^{\text{Bures}}$ is also a von Neumann algebra.
By Proposition~\ref{specsyn}, $\A_0$ and $\A_B$ are the minimal and
maximal $\sigma$-weakly closed $\D$-bimodules with $\supp(\A)$.  It is
somewhat surprising that actually $\A$ is Bures closed, so that
\[\A_0=\A=\A_B=\bimod(\supp(\A)).\] Thus, the class of von Neumann
algebras which lie between $\D$ and $\M$ is a class of
$\D$-bimodules for
which
 Conjecture~\ref{STB} (the spectral theorem
  for bimodules) holds, and for which each element of the class is
  synthetic. 

We now prove these facts, and somewhat more, by extending, and
providing a new proof of, a theorem of Aoi~\cite{AoiCoEqSuInSu}.  Aoi
attributes the statement of his theorem to unpublished work of
C. Sutherland.  Aoi's proof uses the Feldman-Moore formalism and
therefore requires that the von Neumann algebras involved have
separable predual.  Our proof  allows us to
eliminate the separability hypothesis, and also to give a description of
the conditional expectation.

\begin{theorem}[cf. {\cite[Theorem~1.1]{AoiCoEqSuInSu}}]\label{aoi} Let
  $(\M,\D)$ be a Cartan pair, and suppose $\A$ is a von Neumann
  algebra such that $\D\subseteq \A\subseteq \M$.  Then 
 $(\A,\D)$ is a
  Cartan pair, and there exists a unique faithful normal conditional
  expectation $\Phi$ of $\M$ onto $\A$.  In addition, 
$\A$ is a synthetic $\D$-bimodule, 
and the following
  statements hold.
\begin{enumerate} 
\item For each $x\in\fn_\phi$,
   $\Phi(x)\in\fn_\phi$ and
\begin{equation}\label{PhiForm}
\eta_\phi(\Phi(x))=\supp(\A) \eta_\phi(x).
\end{equation}
\item For each $T\in\M$,
\begin{equation}\label{PhiFormB}
\pi_\phi(\Phi(T))P=\supp(\A)\pi_\phi(T) P 
\end{equation} and $\Phi$ is Bures-continuous.
\item If $\E\subseteq \G\N(\A,\D)$ is a maximal $\D$-orthogonal
  family, then for every $x\in\M$, $\Phi(x)$ is the Bures-convergent sum,
\[\Phi(x)=\sum_{u\in \E} uE(u^*x).\]
\end{enumerate}  
\end{theorem}
\begin{proof}
Let $\A_0:=\overline{\spn}^{\sigma\text{-weak}}(\G\N(\M,\D)\cap\A)$.  Then
$\A_0$ is a von Neumann algebra such that 
\[\D\subseteq \A_0\subseteq \A\subseteq \M.\]  
Here is the plan of the proof.  Our first step is to show the
existence of a unique faithful, normal conditional expectation $\Phi$
of $\M$ onto $\A_0$.  Step 2 will show that parts (a) and (b) hold.  We
then show $\A$ is synthetic, that is,
 \begin{equation}\label{ABclosed}
\A_0=\A=
\overline{\A}^{\text{Bures}}.
\end{equation}  Afterwards, we show $(\A,\D)$ is a Cartan pair, and
then conclude the proof by 
verifying (c) holds.

Let $\sigma_t^\phi$ be the modular automorphism group arising from
$\phi$.  Since the span of $\G\N(\M,\D)\cap \A_0$ is $\sigma$-weakly
dense in $\A_0$, Corollary~\ref{modauto} implies that for every
$t\in\bbR$,
\begin{equation}\label{modinv}
\sigma_t^\phi(\A_0)=\A_0.
\end{equation}

Next we  
 show that $\phi|_{\A_0}$ is semi-finite on $\A_0$. 
 Let  
\[\G\Nsf(\M,\D):=\G\N(\M,\D)\cap
\fn_\phi\dstext{and}\G\Nsf(\A_0,\D):=\G\Nsf(\M,\D)\cap \A_0.\]  
Since $\fn_\phi\cap \D$ is
$\sigma$-weakly dense in $\D$, we may find $d_\lambda\in\D\cap
\fn_\phi$ which converges to $I_\D$ $\sigma$-weakly.  Thus for any
$v\in\G\N(\M,\D)\cap \A_0$,  $v=\lim vd_\lambda$ which gives,
\[\overline{\G\Nsf(\A_0,\D)}^{\sigma\text{-weak}} \supseteq \G\N(\M,\D)\cap \A_0.\]
As $\spn(\G\Nsf(\A_0,\D))\subseteq \fn_\phi$, we conclude that $\phi$ is
semi-finite on $\A_0$.

An application of~\cite[Theorem~IX.4.2]{TakesakiThOpAlII} yields the
existence and uniqueness of a normal conditional expectation
$\Phi:\M\rightarrow\A_0$ such that $\phi\circ \Phi=\phi$.  Note that $\Phi$ is
faithful because $\phi$ is.

We now show the formulas~\eqref{PhiForm} and~\eqref{PhiFormB} hold
 for the conditional expectation $\Phi$ just
constructed.  If $x\in\fn_\phi$, then $\phi(\Phi(x)^*\Phi(x))\leq
\phi(\Phi(x^*x))=\phi(x^*x)<\infty$, so $\fn_\phi$ is invariant under
$\Phi$. This calculation also shows the map $\eta_\phi(\fn_\phi)\ni
\eta_\phi(x) \mapsto \eta_\phi(\Phi(x)) $ is norm decreasing on
$\eta_\phi(\fn_\phi)$.  The facts that $\Phi$ is an idempotent linear
map and $\Phi(x)^*=\Phi(x^*)$ for every $x\in\M$ imply that this map
extends to a projection $Q\in\B(\fH_\phi)$ such that
$Q\eta_\phi(x)=\eta_\phi(\Phi(x))$ for every $x\in\fn_\phi$.

Notice $\pi_\phi(\A_0)\eta_\phi(\fn_\phi\cap \D)\subseteq \ran(Q)$.
By definition,
$\ran(\supp(\A_0))=\overline{\pi_\phi(\A_0)\eta_\phi(\fn_\phi\cap
  \D)}$, so $\supp(\A_0)\leq Q$.  On the other hand, for
$x\in\fn_\phi$, $\Phi(x)\in\fn_\phi\cap \A_0$.  By
Lemma~\ref{sotdense}, $\eta_\phi(\fn_\phi\cap \A_0)\subseteq
\overline{\pi_\phi(\A_0)\eta_\phi(\fn_\phi\cap \D)}$, so
$\eta_\phi(\Phi(\fn_\phi))\subseteq \ran(\supp(\A_0)$.  But
$\ran(Q)=\overline{\eta_\phi(\Phi(\fn_\phi))},$ which yields $Q\leq
\supp(\A_0)$.  Thus, $Q=\supp(\A_0)$.  By Lemma~\ref{split}(c),
$\supp(\A_0)=\supp(\A)=\supp(\A_B)$, so~\eqref{PhiForm} holds.

Let $T\in\M$ and $d\in\D\cap\fn_\phi$.  Using~\eqref{PhiForm}, we
have,
\[\pi_\phi(\Phi(T))\eta_\phi(d) =
\eta_\phi(\Phi(Td))=\supp(\A)\eta_\phi(Td) =
\supp(\A)\pi_\phi(T)\eta_\phi(d).\] 
Since $\eta_\phi(\D\cap \fn_\phi)$ is dense in $\ran(P)$, we
obtain~\eqref{PhiFormB}. 

By~\eqref{PhiFormB},  $\Phi$ is Bures continuous.
Let $T\in\overline{\A}^{\text{Bures}}$.  Theorem~\ref{newSTB} ensures
that $T\in\overline{\A_0}^{\text{Bures}}$, so there exists a net
$T_\lambda$ in $\A_0$ which Bures-converges to $T$.  Since $\Phi$ is
Bures continuous,
\[T=\lim_\lambda T_\lambda =\lim_\lambda
\Phi(T_\lambda)=\Phi(\lim_\lambda T_\lambda)= \Phi(T),\] so
$T\in\A_0$.  The equality~\eqref{ABclosed} now follows.

Next we show  $(\A,\D)$ is a Cartan pair.  Obviously, $E|_\A$ is a
faithful normal conditional expectation of $\A$ onto $\D$, so we need
only prove that the span of $\U(\A)\cap \N(\M,\D)$ is $\sigma$-weakly
dense in $\A$.
Since $\A=\A_0$, it suffices to show that 
\begin{equation}\label{ninun}
\G\N(\M,\D)\cap \A\subseteq
\overline{\spn}^{\sigma\text{-weak}}(\G\N(\M,\D)\cap \U(\A)).
\end{equation}

Recall that the maximal ideal space of $\D$, $\hat{\D}$, is a compact,
extremally disconnected space (see
\cite[Theorem~III.1.18]{TakesakiThOpAlI}).  In particular, the Gelfand
transform determines a bijection between the set of projections in
$\D$ and the family of
 clopen subsets of $\hat{\D}$.  Also,
each non-zero $v\in\G\N(\A,\D)$ determines a partial homeomorphism
$\beta_v$ of $\hat{\D}$ with domain
$\{\rho\in\hat{\D}: \rho(v^*v)=1\}$ and
range $\{\rho\in\hat{\D}: \rho(vv^*)=1\}$,
via the formula,
\[\beta_v(\rho)(d):=\rho(v^*dv).\]  When $\beta_v(\rho)=\rho$ for all
$\rho\in\dom(\beta_v)$, it is not difficult to see that $v$ commutes
with $\D$ and hence $v\in\D$.  Finally, when $Q\in\D$ is a projection
such that $Qv^*v\neq 0$, $\beta_{vQ}$ is the restriction of $\beta_v$
to $\{\rho\in\hat{\D}: \rho(Q)=1\}\cap \dom(\beta_v)$.

Given $v\in\G\N(\A,\D)$ with $v\neq 0$, applying a variant of Frol\'{i}k's
Theorem (see~\cite[Proposition~2.7]{PittsStReIn}) to $\beta_v$ yields
projections $Q_0, Q_1, Q_2, Q_3\in\D$ such that $vQ_0\in\D$, $(vQ_j)^2=0$ for
$j=1,2,3$ and
\begin{equation}\label{Frolik}
v=\sum_{j=0}^3 vQ_j.
\end{equation}   

A calculation shows that when $w\in\G\N(\M,\D)\cap \A$ satisfies $w^2=0$, then
\[U:=w+w^*+(I-w^*w-ww^*)\in \U(\A)\cap \G\N(\M,\D)\dstext{and} w=Uw^*w.\] 
  Since $\spn(\U(\D))$ is $\sigma$-weakly dense in $\D$,
we obtain, $w\in \overline{\spn}^{\sigma\text{-weak}}(\G\N(\M,\D)\cap
\U(\A))$. Hence,
\[\overline{\spn}^{\sigma\text{-weak}}(\G\N(\M,\D)\cap\U(\A))\supseteq
\spn\{w\in\G\N(\M,\D)\cap\A: w^2=0\}.\] This together with \eqref{Frolik}
implies that
\[\overline{\spn}^{\sigma\text{-weak}}(\G\N(\M,\D)\cap\U(\A))\supseteq
\spn(\G\N(\M,\D)\cap \A).\]  Thus~\eqref{ninun} holds and $(\A,\D)$ is a
Cartan pair.

To obtain (c), let $\E\subseteq \G\N(\A,\D)$ be
a maximal $\D$-orthogonal family.  For each $u\in\E$,
Corollary~\ref{GNcapS} gives $P_u\leq \supp(\A)$.  For any $u\in\E$,
$x\in\M$ and $d\in\D\cap\fn_\phi$ we obtain,
\begin{align*}
\pi_\phi(uE(u^*\Phi(x)))\eta_\phi(d)&= 
P_u\pi_\phi(\Phi(x))\eta_\phi(d)=P_u\eta_\phi(\Phi(xd))&\text{(now
  apply 
part (a))}\\
&=
P_u\supp(\A)\eta_\phi(xd)=P_u\eta_\phi(xd)=P_u\pi_\phi(x)\eta_\phi(d)\\
&=\pi_\phi(uE(u^*x))\eta_\phi(d).
\end{align*}
This holds for every $d\in\D\cap\fn_\phi$.  Thus (using the fact that 
 $\phi=\phi\circ E$
is faithful) for every $x\in\M$ and $u\in\E$, we have 
\begin{equation}\label{Asupp}
uE(u^*\Phi(x))=uE(u^*x).
\end{equation}
By Proposition~\ref{BConv} applied to the Cartan pair $(\A,\D)$, 
\[\Phi(x)=\sum_{u\in E}uE(u^*\Phi(x))=\sum_{u\in E}uE(u^*x),\] where
the sums are Bures convergent.  This completes the proof.

\end{proof}

\section{An Extension Theorem}

 \stepcounter{subsection}

In this section, we prove our main result about extending isometric
algebra isomorphisms.  We begin with two definitions.

\begin{definition}
\begin{enumerate}
\item Given a Cartan pair $(\M,\D)$, a \textit{Cartan bimodule
algebra} is a $\sigma$-weakly closed subalgebra $\A$ of $\M$
satisfying $\D \subseteq \A$ and which generates $\M$ as a von Neumann
algebra.  We will sometimes write $\A \subseteq (\M,\D)$ to indicate
that $\A$ is a Cartan bimodule algebra for the pair $(\M,\D)$.
\item For $i = 1, 2$, let $\A_i \subseteq (\M_i,\D_i)$ be Cartan
bimodule algebras.  An (algebraic) isomorphism $\theta:\A_1
\rightarrow \A_2$ is a \textit{Cartan bimodule isomorphism} if
$\theta$ is isometric and $\theta(\D_1) = \D_2$.
\end{enumerate}
\end{definition}
\begin{remark}{Remark} In view of Theorem~\ref{aoi}, when $\A$ is a
  $\sigma$-weakly closed subalgebra of $\M$ containing $\D$, $\A$ is a Cartan
  bimodule algebra relative to the Cartan pair, $(W^*(\A),\D)$.
\end{remark}

\begin{lemma}\label{preserves normalizers} For $i = 1, 2$, let $\A_i
\subseteq (\M_i,\D_i)$ be Cartan bimodule algebras and suppose
$\theta:\A_1 \rightarrow \A_2$ is a Cartan bimodule isomorphism.  Then
$\theta(\G\N(\M_1,\D_1) \cap \A_1) = \G\N(\M_2,\D_2) \cap \A_2$.
\end{lemma}

\begin{proof} Let $v \in \G\N(\M_1,\D_1) \cap \A_1$.  Obviously
$\theta(v) \in \A_2$.  For all $h \in \D_1$, we have that
\[ \theta(h)\theta(v) = \theta(hv) = \theta(hvv^*v) = \theta(vv^*hv) =
\theta(v)\theta(v^*hv).
\] Since $\theta|_{\D_1}$ is a $*$-isomorphism, $\theta(v)^*\theta(h)
= \theta(v^*hv)\theta(v)^*$.  Hence, for all $d \in \D_1$, we have
that
\[ \theta(h)\theta(v)\theta(d)\theta(v)^* =
\theta(v)\theta(v^*hv)\theta(d)\theta(v)^* =
\theta(v)\theta(d)\theta(v^*hv)\theta(v)^* =
\theta(v)\theta(d)\theta(v)^*\theta(h),
\] and so $\theta(v)\theta(d)\theta(v)^* \in \M_2 \cap \D_2' = \D_2$.
Likewise $\theta(v)^*\theta(d)\theta(v) \in \D_2$, and so $\theta(v)
\in \N(\M_2,\D_2)$.  We now show that $\theta(v)$ is a partial
isometry.  Note that $p := \theta(v)^*\theta(v)$ belongs to the unit
ball of $\D_2$; we must show that $p$ is a projection.  To do this, we
show that the spectrum of $p$ is $\{0, 1\}$.  If not, let $0 < \lambda
< 1$ belong to the spectrum of $p$, and $\delta > 0$ be such that $0 <
\lambda - \delta < \lambda + \delta < 1$ and let $q$ be the spectral
projection for $p$ corresponding to the interval $(\lambda -
\delta,\lambda + \delta)$.  Then $\theta(v)q \neq 0$, and
$\theta^{-1}(q)$ is a projection in $\D_1$.  Then $0 \neq
v\theta^{-1}(q) \in \G\N(\M_1,\D_1) \cap \A_1$.  As $\theta$ is
isometric,
\[ 1 = \norm{v\theta^{-1}(q)}^2 = \norm{\theta(v)q}^2 = \norm{pq} < 1,
\] which is absurd.  Therefore, the spectrum of $p$ equals $\{0, 1\}$,
so $p$ is a projection.  Hence $\theta(v) \in \G\N(\M_2,\D_2) \cap
\A_2$.  The lemma follows.
\end{proof}

\begin{proposition} \label{Eregular} Let $\A \subseteq (\M,\D)$ be a
Cartan bimodule algebra. Define $\A^0 = \overline{\spn}(\G\N(\M,\D)
\cap \A)$ (norm closure) and $\C = C^*(\G\N(\M,\D) \cap \A)$.  Then
\begin{enumerate}
\item $\C = C^*(\A^0)$ and $\D \subseteq \A^0 \subseteq \C$;
\item $\C = \overline{\spn}(\G\N(\M,\D) \cap \C)$;
\item $\overline{\C}^{\sigma\text{-weak}} = \M$.
\end{enumerate} In particular, the pair $(\C,\D)$ is a $C^*$-diagonal
in the sense of Kumjian~\cite{KumjianOnC*Di}.
\end{proposition}

\begin{proof} (a) and (b) are routine.  We turn now to (c).  Since
$\A^0 \cap \G\N(\M,\D) = \A \cap \G\N(\M,\D)$, we have that
$\overline{\A^0}^{\sigma\text{-weak}} \cap \G\N(\M,\D) = \A \cap
\G\N(\M,\D)$, and so $\supp(\overline{\A^0}^{\sigma\text{-weak}}) =
\supp(\A)$, by Corollary~\ref{GNcapS}.  By
Corollary~\ref{BuresClosure},
\[ \overline{\A^0}^{\text{Bures}} =
\bimod(\supp(\overline{\A^0}^{\sigma\text{-weak}})) =
\bimod(\supp(\A)) = \overline{\A}^{\text{Bures}}.
\] Theorem~\ref{aoi} gives
$\overline{\C}^{\text{Bures}}=\overline{\C}^{\sigma\text{-weak}}$.
Thus,   
\[ \A \subseteq \overline{\A}^{\text{Bures}} =
\overline{\A^0}^{\text{Bures}} \subseteq \overline{\C}^{\text{Bures}}
= \overline{\C}^{\sigma\text{-weak}}=\M,
\] with the last equality holding because $W^*(\A)=\M.$  Hence
$\M = \overline{\C}^{\sigma\text{-weak}}$.

Now (b) says that $(\C,\D)$ is a regular inclusion.  Moreover, as $\D$
is a MASA in $\M$, it is a MASA in $\C$.  Since $\D$ is injective and
$E|_\C$ is a faithful conditional expectation of $\C$ onto $\D$, an
application of~\cite[Theorem~2.10]{PittsStReIn} shows $(\C,\D)$ is a
$C^*$-diagonal.
\end{proof}

\begin{corollary}\label{extend2c*diag} Let $\theta:\A_1 \to \A_2$ be a
Cartan bimodule isomorphism.  Then
\begin{enumerate}
\item there exists a unique $*$-isomorphism $\Theta:\C_1 \to \C_2$
such that $\Theta(x) = \theta(x)$ for all $x \in \A_1^0$ (notation as
in Proposition \ref{Eregular}); and
\item $\Theta(\G\N(\M_1,\D_1) \cap \C_1) = \G\N(\M_2,\D_2) \cap \C_2$.
\end{enumerate}
\end{corollary}

\begin{proof} By Lemma~\ref{preserves normalizers},
$\theta(\G\N(\M_1,\D_1) \cap \A_1) = \G\N(\M_2,\D_2) \cap \A_2$. It
follows that $\theta(\A_1^0) = \A_2^0$.  By
Proposition~\ref{Eregular}, the pair $(\C_i,\D_i)$ is a \cstardiag\
and $C^*(\A_i^0) = \C_i$, for $i = 1, 2$.  An application of
\cite[Theorem~2.16]{PittsNoAlAuCoBoIsOpAl} establishes (a).

Since $\Theta$ is a $*$-isomorphism and $\Theta(\D_1) = \D_2$, (b)
holds.
\end{proof}

The following gives most of Assertion~\ref{Mercer}.

\begin{theorem}\label{pmainMercer} For $i = 1, 2$, let $\A_i \subseteq
(\M_i,\D_i)$ be Cartan bimodule algebras and let $E_i:\M_i \rightarrow
\D_i$ be the faithful normal conditional expectations.  Let
$\theta:\A_1 \rightarrow \A_2$ be a Cartan bimodule isomorphism.  Then
there exists a unique $*$-isomorphism $\overline\theta:\M_1 \to \M_2$
such that
\[ \overline{\theta}|_{\A_1^0} = \theta|_{\A_1^0};
\] and $\theta|_{\A_1^0}$ is a homeomorphism of $(\A_1^0,\tau_B)$ onto
$(\A_2^0,\tau_B)$.

Furthermore, suppose $\omega_1$ is a faithful normal semi-finite
weight on $\D_1$ and let $\omega_2 = \omega_1 \circ
(\theta|_{\D_1})^{-1}$.  Set $\phi_i := \omega_i \circ E_i$ and let
$(\pi_{\phi_i},\fH_i,\eta_{\phi_i})$ be the semicyclic representation
of $\M_i$ corresponding to $\phi_i$.  Then there exists a unitary
$U:\fH_1 \rightarrow \fH_2$ such that for every $X \in \M_1$,
\[ U\pi_{\phi_1}(X) = \pi_{\phi_2}(\overline{\theta}(X))U.
\]
\end{theorem}

\begin{proof} We use the notation of Corollary~\ref{extend2c*diag}.
Since $(\C_i,\D_i)$ are \cstardiag s, $\C_i$ has the extension
property relative to $\D_i$.  In particular, the expectations
$E_i|_{\C_i}$ are unique.  Since $\Theta \circ E_1|_{\C_1} \circ
\Theta^{-1}$ is a conditional expectation of $\C_2$ onto $\D_2$, we
obtain
\[ E_2|_{\C_2} \circ \Theta = \Theta \circ E_1|_{\C_1}.
\] Hence
\[ (\omega_2 \circ E_2)|_{\C_2} = (\omega_1 \circ E_1)|_{\C_1} \circ
\Theta^{-1}.
\]

For $Y \in \C_1$ and $d \in \fn_{\phi_1} \cap \D_1$, we have
\[ \norm{\pi_{\phi_2}(\Theta(Y))\eta_{\phi_2}(\theta(d))}^2 =
\omega_2(E_2(\Theta(d^*Y^*Yd))) = \omega_1(E_1(d^*Y^*Yd)) =
\norm{\pi_{\phi_1}(Y)\eta_{\phi_1}(d)}^2.
\] As $\pi_{\phi_i}(\C_i)\eta_{\phi_i}(\fn_{\phi_i} \cap \D_i)$ is
dense in $\fH_i$ by Lemma~\ref{sotdense}, we find that the map
$\pi_{\phi_1}(Y)\eta_{\phi_1}(d) \mapsto
\pi_{\phi_2}(\Theta(Y))\eta_{\phi_2}(\theta(d))$ extends to a unitary
$U \in \B(\fH_1, \fH_2)$.  Moreover, for any $X \in \C_1$, we obtain
\[ U\pi_{\phi_1}(X) = \pi_{\phi_2}(\Theta(X))U.
\]

For $X \in \M_1$ we now define
\[ \overline{\theta}(X) := \pi_{\phi_2}^{-1}(U\pi_{\phi_1}(X)U^*).
\] Then $\overline\theta$ is a $*$-isomorphism of $\M_1$ onto $\M_2$
and by construction, $\overline\theta|_{\A_1^0} = \Theta|_{\A_1^0} =
\theta|_{\A_1^0}$.

The uniqueness of $\overline\theta$ follows from the facts that $\C_i$
are $\sigma$-weakly dense in $\M_i$ and $\Theta$ is the unique
extension of $\theta|_{\A_1^0}$ to a $*$-isomorphism of $\C_1$ onto
$\C_2$.

It is easy to see that $\overline{\theta} \circ E_1 = E_2 \circ
\overline{\theta}$, which implies $\overline{\theta}$ and
$(\overline{\theta})^{-1}$ are Bures continuous.  Thus the restriction
of $\overline\theta$ to $\A_1^0$ is a Bures homeomorphism onto
$\A_2^0$.
\end{proof}

We now strengthen Theorem~\ref{pmainMercer} by showing that when
$\theta$ is $\sigma$-weakly continuous, $\overline{\theta}|_{\A_1} =
\theta$.  (Note: If we knew that $\A_1^0$ was $\sigma$-weakly dense in
$\A_1$, this would be trivial.  Unfortunately, all we know is that
$\A_1^0$ is Bures dense in $\A_1$.)  We require some preparation.  The
notation will be as in Theorem~\ref{pmainMercer}.

\begin{lemma}\label{wellbeh} For $i = 1, 2$, let $\A_i \subseteq
(\M_i,\D_i)$ be Cartan bimodule algebras and let $E_i:\M_i \rightarrow
\D_i$ be the faithful normal conditional expectations.  Let
$\theta:\A_1 \rightarrow \A_2$ be a $\sigma$-weakly continuous Cartan
bimodule isomorphism.  If $x \in \A_1$ and $v \in \G\N(\M_1,\D_1)$,
then
\begin{equation}\label{wb} \theta(vE_1(v^*x)) =
\overline{\theta}(v)E_2(\overline{\theta}(v)^*\theta(x)).
\end{equation}
\end{lemma}

Before giving the proof, notice that Lemma~\ref{split}(b) gives
$vE(v^*x) \in \A_1$, so the left side of \eqref{wb} is defined.

\begin{proof} Let $x \in \A_1$ and $v \in \G\N(\M_1,\D_1)$.  By
Lemma~\ref{uniquece},
\begin{equation}\label{ucorc} \{vE_1(v^*x)\} = v\D_1 \cap
\overline{\text{co}}^{\sigma\text{-weak}}\{vUv^*xU^*: U \in
\U(\D_1)\}.
\end{equation}

We then have,
\begin{align*} \theta(vE_1(v^*x)) &\in
\theta(\overline{\text{co}}^{\sigma\text{-weak}}\{vUv^*xU^*: U \in
\U(\D_1)\})\\ &\subseteq
\overline{\text{co}}^{\sigma\text{-weak}}\{\theta(vUv^*)\theta(x)\theta(U^*):
U \in \U(\D_1)\}\\ &=
\overline{\text{co}}^{\sigma\text{-weak}}\{\overline{\theta}(vUv^*)\theta(x)\theta(U)^*:
U \in \U(\D_1)\}\\ &=
\overline{\text{co}}^{\sigma\text{-weak}}\{\overline{\theta}(v)\theta(U)\overline{\theta}(v)^*\theta(x)\theta(U)^*:
U \in \U(\D_1)\}\\ &=
\overline{\text{co}}^{\sigma\text{-weak}}\{\overline{\theta}(v)W\overline{\theta}(v)^*\theta(x)W^*:
W \in \U(\D_2)\}.
\end{align*} Since $vE_1(v^*x) \in \A_1^0$
(cf. Lemma~\ref{polarnormal}), we have $\theta(vE_1(v^*x)) =
\overline{\theta}(vE_1(v^*x)) \in \overline{\theta}(v)\D_2$.  Thus,
\[ \theta(vE_1(v^*x)) \in \overline{\theta}(v)\D_2 \cap
\overline{\text{co}}^{\sigma\text{-weak}}
\{\overline{\theta}(v)W\overline{\theta}(v)^*\theta(x)W^*: W \in
\U(\D_2)\} =
\{\overline{\theta}(v)E_2(\overline{\theta}(v)^*\theta(x))\}.
\] The lemma follows.
\end{proof}

\begin{theorem}\label{mainMercer} In addition to the hypotheses of
Theorem~\ref{pmainMercer}, assume $\theta$ is $\sigma$-weakly
continuous.  Then
\[ \theta = \overline{\theta}|_{\A_1}.
\]
\end{theorem}

\begin{proof} Let $\E \subseteq \G\N(\M_1,\D_1)$ be a maximal
$\D_1$-orthogonal set.  Then $\overline{\theta}(\E) \subseteq
\G\N(\M_2,\D_2)$ is a maximal $\D_2$-orthogonal set.

Let $X \in \A_1$ and suppose $F \subseteq \E$ is a finite set.  Then,
with the notation of Proposition~\ref{BConv} and using Lemma~\ref{wellbeh},
we have
\[ \theta(X_F) = \sum_{v \in F} \theta(vE_1(v^*X)) = \sum_{v \in F}
\overline{\theta}(v)E_2(\overline{\theta}(v)^*\theta(X)).
\] It then follows from Proposition~\ref{BConv} that $\theta(X_F)$ Bures
converges to $\theta(X)$.  On the other hand, since $X_F \in \A_1^0$,
we have $\theta(X_F) = \overline{\theta}(X_F)$.  As we noted in the
proof of Theorem~\ref{pmainMercer}, $\overline{\theta}$ is Bures
continuous.  Therefore,
\[ \overline{\theta}(X) = \text{Bures-}\lim \overline{\theta}(X_F) =
\text{Bures-}\lim \theta(X_F) =\theta(X).
\]
\end{proof}

\begin{remark}{Remark}\label{finalremark}
  Without a continuity hypothesis, we have been unable to
  obtain~Assertion~\ref{Mercer}, even when the Cartan pairs
  $\A_i\subseteq (\M_i,\D_i)$ are assumed synthetic.  Suppose $\A_i$
  are synthetic.  With the notation of Theorem~\ref{pmainMercer}, let
  $\alpha:=\overline{\theta}^{-1}\circ\theta$.  The hypothesis of
  synthesis implies $\alpha$ is an isometric automorphism of $\A_1$
  such that $\alpha|_{\A_1^0}=\text{id}|_{\A_1^0}.$ We have not been
  able to show $\alpha=\text{id}_{\A}$ without making a continuity
  hypothesis, and we suspect such a hypothesis may in general be
  necessary.

\end{remark}

\providecommand{\bysame}{\leavevmode\hbox to3em{\hrulefill}\thinspace}
\providecommand{\MR}{\relax\ifhmode\unskip\space\fi MR }

\providecommand{\MRhref}[2]{%
  \href{http://www.ams.org/mathscinet-getitem?mr=#1}{#2}
}
\providecommand{\href}[2]{#2}
\providecommand{\Zbl}[1]{#1.}

\def\cprime{$'$}

\end{document}